\documentclass{article}

\usepackage{rotating,color}
\usepackage{amsthm,amssymb}
\usepackage{subfig}
\usepackage{amsmath,hyperref}
\usepackage{pdflscape}
\usepackage{verbatim}

\usepackage[ruled,vlined,linesnumbered]{algorithm2e}

\SetAlgorithmName{Algorithm Schema}{\autoref}{List of Algorithms}

\SetKwFor{ParallelForEach}{for each}{in parallel do}{endfor}

\newcommand{\K}[0]{{\mathcal C}}

\DeclareMathOperator*{\diag}{diag}

\newcommand{\st}{\;:\;}

\newcommand{\R}{\mathbf{R}}

\usepackage{epstopdf}% To incorporate .eps illustrations using PDFLaTeX, etc.
 
\theoremstyle{plain}% Theorem-like structures
\newtheorem{theorem}{Theorem}[section]

\newtheorem{proposition}[theorem]{Proposition}

\theoremstyle{definition}

\newtheorem{example}[theorem]{Example}

\theoremstyle{remark}
\newtheorem{remark}{Remark}

\usepackage[colorinlistoftodos,bordercolor=orange,backgroundcolor=orange!20,linecolor=orange,textsize=scriptsize]{todonotes}

\usepackage{dcolumn}

\newcolumntype{d}[1]{D{.}{\cdot}{#1} }

\newcommand{\inTR}[1]{}
\newcommand{\rank}[0]{\textrm{rank}}
\newcommand{\trace}[0]{\textrm{tr}}

\SetCommentSty{\small\sf}

\newcommand{\EDIT}[1]{#1}

\newcommand{\tr}{\textrm{tr}}

\renewcommand{\st}{\mbox{s.t.} \ }
\newcommand{\RK}{R^{(k)}}
\newcommand{\setN}{N}

\usepackage{algorithmic}

\begin{document}

%\jvol{00} \jnum{00} \jyear{2015} \jmonth{April}

%\articletype{GUIDE}

\title{A Low-Rank Coordinate-Descent Algorithm   for Semidefinite Programming Relaxations  of Optimal Power Flow
  }

\author{
Jakub Mare\v{c}ek
%\textsuperscript{a}$^{\ast}$
%\thanks{$^\ast$Corresponding author. Email: jakub.marecek@ie.ibm.com}
and Martin Tak\'a\v{c}
%\textsuperscript{b}}
%\affil{\textsuperscript{a}IBM Research -- Ireland, IBM Campus Damastown, D15, Ireland\\
%\textsuperscript{b}Lehigh University, Bethlehem, PA, USA}
%\received{v4.0 released April 2015}
}

\maketitle

\begin{abstract}
The alternating-current optimal power flow \EDIT{(ACOPF)} is one of the best known non-convex non-linear optimisation problems. 
\EDIT{We present a} novel re-formulation of ACOPF, \EDIT{which is based on lifting the rectangular power-voltage rank-constrained formulation,} 
and makes it possible to derive \EDIT{alternative} SDP relaxations. 
For those, we develop a first-order method based on the \EDIT{parallel} coordinate descent \EDIT{with a novel closed-form step based on roots of cubic polynomials}.
%\EDIT{quite unlike second-order interior-point methods used to solve convexifications recently.}   

%A novel rank-constrained re-formulation of alternating-current optimal power flow problem
%makes it possible to derive novel semidefinite programming (SDP) relaxations. For those, we 
%develop a solver, which is often as widely used methods with no convergence guarantees, 
%within the same accuracy. 

%for the standard datasets, 
%In line with previous observations, the objective value of the optimum of the SDP relaxation 
%is often considerably better than the one obtained with Matpower, e.g., by 38.19\% on case9mod. 
%\todo[inline]{I think we have to explude the MOD dataset....}

%and by 7.97\% on case30Q.
%\rr{we spoke that the case30Q should be exluded, not sure if some similar issue are with case9mod as well??}
% I'll check that, again, but afaik case9mod is "for real".
\end{abstract}

\section{Introduction}
 
%\todo{I hope to: add results on at least one of the Polish instances, add tables with properties of the iterates (feasibility measures for the SDP) }

Alternating-current optimal power flow problem (ACOPF) is one of the best known 
non-linear optimisation problems \cite{wood1996power}.
Due to its non-convexity, deciding feasibility is NP-Hard 
even for a tree network with fixed voltages \cite{Lehmann2015}.
Still, there has been much recent progress \EDIT{\cite{6756976,6815671}}:  
Bai et al. \cite{Bai2008} introduced a semidefinite programming (SDP) relaxation, which turned out to be \EDIT{particularly} opportune.
Lavaei and Low \cite{lavaei2012zero}
have shown \EDIT{that} the SDP relaxation produces exact solutions,
under certain spectral conditions.
More generally,
it can be strengthened so as to obtain a hierarchy of SDP relaxations
whose optima are asymptotically converging to the true optimum of ACOPF
\cite{Ghaddar2015,RTE2013,josz2015moment}.
Unfortunately, the run-times of \EDIT{even} the best-\EDIT{performing} solvers for the SDP relaxations
\EDIT{\cite{sturm1999,permenter2015solving} remain} % several orders of magnitude 
much higher than that of commonly used methods without global convergence guarantees such as Matpower \cite{Matpower}.

Following a brief overview of our notation, we introduce a novel rank-constrained  reformulation 
of the ACOPF problem \EDIT{in Section \ref{sec:reform}}, where all constraints, except for the rank constraint, are coordinate-wise.
\EDIT{Based on this reformulation, we derive novel SDP relaxations.}
%Notice that some assumptions are necessary, considering the problem is NP-hard, 
%  and we are hence guaranteed to find only a stationary point, %the coordinate-wise local optimum 
%  in general.
%However, if we assume that the SDP relaxation has a rank one solution, we provide a theorem, which can be used to check if we have found a global optimum.}
%\rr{Subsequently, we present a method, which solves the relaxation with progressively higher rank,
%testing the stationary points for global optimality.}
Next, we present a parallel coordinate-descent algorithm \EDIT{in Section \ref{sec:algo}}, which solves a sequence of
convex relaxations with coordinate-wise constraints, using a novel closed-form step considering roots of cubic polynomials.
\EDIT{
We can show:
\begin{itemize}
\item the algorithm converges to the exact optimum of the non-convex problem, when the optimum of the non-convex problem coincides with the optimum of the novel SDP relaxations and certain additional assumptions
are satisfied, as detailed in Section \ref{sec:analysis}
\item the algorithm suggests a strengthening of the novel SDP relaxations is needed, whenever it detects the optimum of the non-convex problem has not been found, using certain novel sufficient conditions
of optimality
%We also provide sufficient and necessary conditions for a stationary point to be an optimum.
\item our pre-liminary implementation reaches a precision comparable to the default settings of Matpower \cite{Matpower}, the commonly used interior-point method without global convergence guarantees,
on certain well-known instances including the IEEE 118 bus test system, within comparable times, as detailed in Section \ref{sec:numerical}
\end{itemize}
}
\EDIT{although much more work remains to be done, especially with focus on large-scale instances. Also, as with most solvers, the proposed assumes feasibility and does not certify infeasibility, when encountered. }
%still often providing much better solutions, in terms of the objective value.
The proofs of convergence rely on the work of Burer and Monteiro \cite{BurerMonteiro2005},
Grippo et al. \cite{grippo2009necessary},
and our earlier work \cite{RT2015,marecek2015distributed}.

%\redd{our current rank 1 primal}
%
%\redd{katya's BCD for dual with A as low rank}
%
%\redd{katya's BCD for dual with A as big rank}
%
%\redd{iteration complexity to stationary point for primal rank-1}
%
%\redd{proof that if $r>\mathcal{O}(\sqrt m)$ then every local minimum is global
%for rank r minimization problem}
  
\section{The Problem}

%\subsection{Formulation}

Informally, within the optimal power flow problem, one aims to decide where to generate power,
such that the demand for power is met and costs of generation are minimised.
In the alternating-current model, one considers the complex voltage, complex current, 
and complex power, although one may introduce decision variables representing only a 
subset thereof.
The constraints are non-convex in the alternating-current model, 
and a particular care hence needs to be taken when modelling those, 
and designing the solvers to match.

Formally, we start a network of buses $\setN$, connected by branches $L \subseteq \setN  \times \setN$, modeled as $\Pi$-equivalent circuits,
with the input comprising also of:
\begin{itemize}
\item $G \subseteq \setN$, which are the generators, \EDIT{with the associated coefficients 
$c^0_k, c^1_k, c^2_k$ of the quadratic cost function at $k \in G$},
\item $P^d_k + jQ^d_k$, which are the active and reactive loads at each bus $k \in \setN$, % (demand) 
\item $P_k^{\min}$, $P_k^{\max}$, $Q_k^{\min}$ and $Q_k^{\max}$, which are the limits on active and reactive generation capacity at bus $k \in G$, where $P_k^{\min}=P_k^{\max}=Q_k^{\min}=Q_k^{\max} =0$ for all $k \in \setN \setminus G$,
\item $y \in \mathbb{C}^{|N|\times|N|}$, which is the network admittance matrix capturing the value of the shunt element $\bar{b}_{lm}$ and series admittance $g_{lm}+jb_{lm}$ at branch $(l,m) \in L$,
\item $V_k^{\min}$ and $V_k^{\max}$, which are the limits on the absolute value of the voltage at a given bus $k$,
\item $S_{lm}^{\max}$, which is the limit on the absolute value of the apparent power of a branch $(l,m)\in L$.
\end{itemize}
In the rectangular power-voltage formulation, the variables are:
\begin{itemize}
\item $P^g_k + j Q^g_k$, which is the power generated at bus $k \in G$,
\item $P_{lm} + j Q_{lm}$, which is the power flow along branch $(l,m) \in L$,
\item $\Re{V_k}+ j \Im{V_k}$, which is the voltage at each bus $k \in N$.
\end{itemize}
The power-flow equations at generator buses $k \in G$ are:
\begin{align}
P^g_k &= P_k^d+\Re{V_k} \sum_{i=1}^n ( { \Re{y_{ik}} \Re{V_i} - \Im{y_{ik}} \Im{V_i} })      + \Im{V_k} \sum_{i=1}^n ({ \Im{y_{ik}} \Re{V_i} + \Re{y_{ik}} \Im{V_i} }), \label{eqn:Pk} \\
Q^g_k &= Q_k^d+\Re{V_k} \sum_{i=1}^n ({ - \Im{y_{ik}} \Re{V_i} - \Re{y_{ik}} \Im{V_i} }  )     + \Im{V_k} \sum_{i=1}^n ({ \Re{y_{ik}} \Re{V_i} - \Im{y_{ik}} \Im{V_i} }), \label{eqn:Qk}
\end{align}
\EDIT{while at all other buses $k \in N \setminus G$ we have:}
\begin{align}
0 &= P_k^d+\Re{V_k} \sum_{i=1}^n ( { \Re{y_{ik}} \Re{V_i} - \Im{y_{ik}} \Im{V_i} })      + \Im{V_k} \sum_{i=1}^n ({ \Im{y_{ik}} \Re{V_i} + \Re{y_{ik}} \Im{V_i} }), \\
0 &= Q_k^d+\Re{V_k} \sum_{i=1}^n ({ - \Im{y_{ik}} \Re{V_i} - \Re{y_{ik}} \Im{V_i} }  )     + \Im{V_k} \sum_{i=1}^n ({ \Re{y_{ik}} \Re{V_i} - \Im{y_{ik}} \Im{V_i} }).
\end{align}
\EDIT{Additionally, the power flow at branch $(l,m) \in L$ is expressed as}
\begin{align}
P_{lm} &= b_{lm} ( \Re{V_l} \Im{V_m} - \Re{V_m} \Im{V_l}) \label{eqn:Plm}       + g_{lm}( \Re{V_l}^2 +\Im{V_m}^2- \Im{V_l}, \Im{V_m}- \Re{V_l} \Re{V_m}),   \\
Q_{lm}  &= b_{lm} ( \Re{V_l} \Im{V_m} - \Im{V_l} \Im{V_m}-\Re{V_l}^2 -\Im{V_l}^2) \notag \\
    & \ + g_{lm}( \Re{V_l}\Im{V_m}- \Re{V_m} \Im{V_l}- \Re{V_m} \Im{V_l})   -\frac{\bar b_{lm}}{2}(\Re{V_l}^2 +\Im{V_l}^2). \label{eqn:Qlm}
\end{align}
\EDIT{Considering the above,} the alternating-current optimal power flow (ACOPF) is:
%\todo[inline]{this sentence sounds weird to me:)}
%Let $x$ be a vector of variables defined as $x: =[\Re{V}_k \quad \Im{V}_k ]^T$, and let the cost of power generation be  $\sum_{k\in G} f_k(P_k^g)$ where $f_k(P_k^g)= c^2_k (P_k^g)^2 + c^1_k P_k^g + c^0_k$, with $c^2_k, c^1_k, c^0_k$ non-negative. The classical OPF problem is presented below:
%{\small
\begin{align*}
\min &\sum_{k\in G}\left( c^2_k (P_k^g + P_k^d )^2 + c^1_k (P_k^g + P_k^d) + c^0_k \right)\hspace{-0.2cm}& \tag*{[ACOPF]} \label{ACOPF} \\
\st 
& P^{\min}_k \le P^g_k + P^d_k  \le P^{\max}_k &  \forall k  \in G\\
%& P_k = P^{d}_k &  \forall k  \in N \setminus G\\
& Q^{\min}_k \le Q^g_k + Q^d_k \le Q^{\max}_k & \forall k  \in G \\
%& Q_k = Q^{d}_k &  \forall k \in N \setminus G\\
& (V_k^{\min})^2 \le\Re{V_k}^2+\Im{V_k}^2 \le ( V_k^{\max})^2 \hspace{-0.3cm}& \forall k  \in N \\
& P_{lm}^2+Q_{lm}^2 \le (S_{lm}^{\max})^2 & \forall (l,m) \in L\\
& \eqref{eqn:Pk}-\eqref{eqn:Qlm},&
\end{align*}
\EDIT{where $c^2_k$ is the coefficient of the leading term of the quadratßic cost function at generator $k$.}
% with the apparent power $S_{lm}=P_{lm}+jQ_{lm}$ on the line $(l,m) \in E$ represented in terms of voltages.

% \begin{align}
% & P^{\min}_k \le P^g_k \le P^{\max}_k & \\
% & Q^{\min}_k \le Q^g_k \le Q^{\max}_k & \\
% & (V_k^{\min})^2 \le\Re{V_k}^2+ \Im{V_k}^2 \le ( V_k^{\max})^2 & \forall k  \in N\\
% & P_{lm}^2+Q_{lm}^2 \le S_{l,m}^{\max} & \forall (l,m) \in E
% \end{align}
In line with recent work \cite{Bai2008,lavaei2012zero,molzahn2011,Ghaddar2015}, let $e_k$ be the $k^{th}$ standard basis vector in $\mathbb{R}^{2|N|}$ and define:
\begin{align*}
y_k&=e_ke_k^T y, \\
y_{lm}&=(j\frac{\bar b_{lm}}{2}+g_{lm}+jb_{lm})e_le_l^T - (g_{lm}+jb_{lm})e_le_m^T, \\
Y_k&= \frac{1}{2} \left[ \begin{matrix} \Re(y_k  + y_k^T)  &\Im(y_k^T -y_k)  \\
   \Im(y_k  - y_k^T)& \Re(y_k  + y_k^T) \end{matrix}  \right], \\
\bar{Y}_k&= -\frac{1}{2} \left[ \begin{matrix} \Im(y_k  + y_k^T)& \Re(y_k -y_k^T) \\
   \Re(y_k^T  - y_k) & \Im(y_k  + y_k^T) \end{matrix}  \right], \\
   M_k&=  \left[ \begin{matrix}e_ke_k^T & 0  \\
   0 & e_ke_k^T \end{matrix}  \right], \\
   Y_{lm}&= \frac{1}{2} \left[ \begin{matrix} \Re(y_{lm}  + y_{lm}^T)& \Im(y_{lm}^T -y_{lm}) \\
   \Im(y_{lm} - y_{lm}^T) & \Re(y_{lm}  + y_{lm}^T) \end{matrix}  \right], \\
\bar Y_{lm}  &= \frac{{ - 1}}{2}\left[ {\begin{array}{*{20}c}
   {\Im ( y_{lm}  + y_{lm}^T ) } & {\Re ( y_{lm}  - y_{lm}^T ) }  \\
   {\Re ( y_{lm}^T  - y_{lm} ) } & {\Im ( y_{lm}  + y_{lm}^T ) }  \\
\end{array}} \right].
  \end{align*}
Using these matrices, we can rewrite \ref{ACOPF} as a real-valued polynomial optimization problem of degree four
in variable $x \in \R^{2|N|}$ comprising of $\Re{V}_k \in \R^{|N|}$ and $\Im{V}_k \in \R^{|N|}$, stacked:
%{\small
\begin{align}
\min &\sum_{k\in G} \left(c^2_k(\text{tr}(Y_k xx^T) + P_k^d)^2+c^1_k(\text{tr}(Y_k xx^T) + P_k^d)+c^0_k\right) \tag*{[PP4]} \label{PP4}                               & \\
\st 
&P_k^{\min}\leq \text{tr}(Y_kxx^T) + P_k^d \leq P_k^{\max}                  & \forall k \in G  \label{eqn:pp1} \\
&P_k^{d} = \text{tr}(Y_kxx^T)                                                    & \forall k \in N \setminus G \label{eqn:pp2}  \\
& Q_k^{\min}\leq \text{tr}(\bar{Y}_kxx^T) + Q_k^d \leq Q_k^{\max}                  & \forall k \in G  \\
&Q_k^{d} = \text{tr}(Y_kxx^T)                                                    & \forall k \in N \setminus G  \\
&(V_k^{\min})^2 \leq \text{tr}(M_kxx^T) \leq  (V_k^{\max})^2                      & \forall k \in N \label{eqn:pp3} \\
&(\text{tr}(Y_{lm}xx^T))^2+(\text{tr}(\bar{Y}_{lm}xx^T))^2 \leq (S_{lm}^{\max})^2 & \forall (l,m) \in L, \label{eqn:pp4}
\end{align}
%}
Henceforth, we use $n$ to denote $2|N|$, i.e., the dimension of the real-valued problem \ref{PP4}.
%quadratic objective function

%often is the cost of power generation where 
%$$f_k(x):=\left(c^2_k(P_k^d + \text{tr}(Y_kxx^T))^2+c^1_k(P_k^d + \text{tr}(Y_kxx^T))+c^0_k\right).$$ Constraints \eqref{eqn:pp1} and \eqref{eqn:pp2} impose a limitation on the active and reactive power. Constraints \eqref{eqn:pp3} restrict the voltage on a given bus. Constraints \eqref{eqn:pp4} limit the apparent power flow at each end of a given line. 
Further, the problem can be lifted to obtain a rank-constrained problem in $W=xx^T \in \R^{n \times n}$:
\begin{align}
\min &\sum_{k\in G} f_k(W)  \tag*{[R1]} \label{R1}                   & \\
\st 
&P_k^{\min} \leq \text{tr}(Y_kW) + P_k^d \leq P_k^{\max}             & \forall k \in G \label{R1-Constraint1} \\
& P_k^{d} = \text{tr}(Y_kW)                                         & \forall k \in N \setminus G \\
& Q_k^{\min}\leq \text{tr}(\bar{Y}_kW) +Q_k^d  \leq Q_k^{\max}       & \forall k \in G \\
& Q_k^d = \text{tr}(\bar{Y}_kW)                                      & \forall k \in N \setminus G  \\
&(V_k^{\min})^2 \leq \text{tr}(M_kW) \leq  (V_k^{\max})^2            & \forall k \in N \\
&(\text{tr}(Y_{lm}W))^2+(\text{tr}(\bar{Y}_{lm}W))^2 \leq (S_{lm}^{\max})^2 & \forall (l,m) \in L \label{R1-Constraint4} \\
& W \succeq 0, \quad \rank(W) = 1,                                   &
\end{align}
for a suitable definition of $f_k$. This problem \ref{R1} is still NP-Hard, but where one can drop the rank constraint to obtain a strong and efficiently solvable SDP relaxation:
\begin{align}
\min &\sum_{k\in G} f_k(W) \; \st \; (\ref{R1-Constraint1}-\ref{R1-Constraint4}), \quad W \succeq 0, \tag*{[LL]} \label{LL}
\end{align}
as suggested by \cite{Bai2008}. Lavaei and Low \cite{lavaei2012zero} studied the conditions, under which the relaxation \ref{LL} (Optimization 3 of \cite{lavaei2012zero})
is equivalent to \ref{R1}. 
We note that for traditional solvers \cite{molzahn2011,lavaei2012zero,Ghaddar2015}, the dual of the relaxation (Optimization 4 of \cite{lavaei2012zero})
is much easier to solve than \ref{LL}\EDIT{, as documented in Table \ref{tabComparisonWithGP}}.
Ghaddar et al.\ \cite{Ghaddar2015} have shown the relaxation \ref{LL} to be equivalent to a first-level $[OP_4-H_1]$  
of a certain hierarchy of relaxations,
and how to obtain the solution to \ref{R1} under much milder conditions than those of \cite{lavaei2012zero}.

\section{The Reformulation}
\label{sec:reform}

%We return to the non-convex optimisation over the vectors \eqref{eqn:1}--\eqref{eqn:4},
The first contribution of this paper is a lifted generalisation of \ref{R1}:
 
\begin{align}
\min & %\ \ F(W):=
\sum_{k\in G} f_k(W) \tag*{[R$r$BC]}
\label{RrBC} \\
\st & t_k = \text{tr}(Y_kW)                               & \forall k \in N \label{eq:bc0}\\ 
P_k^{\min} - P_k^d \leq \;& t_k \leq P_k^{\max} - P_k^d     & \forall k \in G \label{eq:bc1}\ \\
%& P_k^{d} = \text{tr}(Y_kW)   & \forall k \in N \setminus G  \\
& P_k^{d} = t_k                                           & \forall k \in N \setminus G  \\
& g_k = \text{tr}(\bar{Y}_kW)                             & \forall k \in N   \\ 
Q_k^{\min} - Q_k^d \leq \; & g_k \leq Q_k^{\max} - Q_k^d     & \forall k \in G \label{eqn:bc2} \\
%& Q_k^{d} = \text{tr}(\bar{Y}_kW)                         & \forall k \in N \setminus G  \\
& Q_k^{d} = g_k                                           & \forall k \in N \setminus G  \\
& h_k = \text{tr}(M_kW)                                   & \forall k \in N \\ 
(V_k^{\min})^2 \leq \; & h_k \leq  (V_k^{\max})^2            & \forall k \in N \label{eqn:bc3} \\
& u_{lm} =  \text{tr}(Y_{lm}W)                            & \forall (l,m) \in L \\ 
& v_{lm} = \text{tr}(\bar{Y}_{lm}W)                       & \forall (l,m) \in L \label{eqn:barLimLM}\\
& z_{lm} = (u_{lm})^2+(v_{lm})^2                          & \forall (l,m) \in L \\   
& z_{lm}  \leq (S_{lm}^{\max})^2                          & \forall (l,m) \in L \label{eqn:bc4} \\
& W \succeq 0, \rank(W) \le r. \label{eqn:bc5}
\end{align}
 
There, we still have:

\begin{proposition}[Equivalence]
\label{prop:1}
[R1BC] is equivalent to \ref{PP4}.
\end{proposition}

%\begin{proof}
%\end{proof}

Even in the special case of $r = 1$, however, we have lifted the 
problem to a higher dimension by adding variables $t, g, h, u,v,z$, which are box-constrained functions of $W$.

Subsequently, we make four observations to motivate our approach to solving the \ref{RrBC}:
\begin{itemize} 
\item[$O_1$:] \EDIT{constraints \eqref{eq:bc1}, \eqref{eqn:bc2}, \eqref{eqn:bc3}, and \eqref{eqn:bc4} 
 are box constraints, while the remainder of (\ref{eq:bc0}--\ref{eqn:bc4}) are linear equalities}
\item[$O_2$:] using elementary linear algebra:
\begin{align*}\{ W \in \mathcal{S}^n \; \textrm{s.t.} \; W \succeq 0, \quad \rank(W) \le r \} \\ 
= \{ RR^T \; \textrm{s.t.} \; R \in \R^{n \times r} \},
\end{align*}
where $\mathcal{S}^n$ is the set of symmetric $n \times n$ matrices.
%if $  W\in \R^{n \times n}$ such that $W \succeq 0$ and $\rank(W) = r$, then it can be written as $W = R R^T$
%where $R \in \R^{n \times r}$
\item[$O_3$:] if $\rank(W^*) > 1$ for the optimum $W^*$ of \ref{LL},
there are no known methods for extracting the global optimum of \ref{R1} from $W$, except for \cite{Ghaddar2015}.
\item[$O_4$:] zero duality gap at any SDP relaxation in the hierarchy of \cite{Ghaddar2015} does not guarantee the solution of the SDP relaxation is exact for \ref{PP4}, c.f. \cite{Josz2015}.
\end{itemize}

Note that Lavaei and Low \cite{lavaei2012zero} restate the condition in Observation $O_3$ 
in terms of ranks using a related relaxation (Optimization 3), where the rank has to be strictly larger than 2.

\section{The Algorithm}
\label{sec:algo}

%Due to the fact that \eqref{RrBC} is a constrained optimization problem, 
%we have to have a strategy to be able to handle them.

Broadly speaking, we use the well-established Augmented Lagrangian approach 
\cite{powell1978fast,conn1991globally}, with a low-rank twist \cite{BurerMonteiro2005},
and a parallel coordinate descent with a closed-form step.
Considering Observation $O_2$, we replace variable $W \in \R^{n \times n}$ by $RR^T \in \R^{n \times n}$ for
 $R \in \R^{n \times r}$ % in solving \eqref{RrBC} 
to obtain the following augmented Lagrangian function:
%{\small
\begin{align}
\label{AugLagrangian}
&\mathcal{L}(R,
t,h,g,u,v,z,\lambda^t, \lambda^g, \lambda^h,
\lambda^{u},\lambda^{v},\lambda^z):= \\
& \sum_{k\in G} f_k(RR^T) \notag
\\&- 
\sum_k \lambda^t_k (t_k  - \tr(Y_k RR^T)) 
 + \tfrac{\mu}{2} \sum_k   (t_k  - \tr(Y_k RR^T)) ^2 \notag
\\&- 
\sum_k \lambda^g_k (g_k - \tr(\bar Y_k RR^T)) 
 + \tfrac{\mu}{2} \sum_k   (g_k - \tr(\bar Y_k RR^T)) ^2 \notag
\\&- 
\sum_k \lambda_k^h  (h_k  - \tr(M_k RR^T)) 
 + \tfrac{\mu}{2} \sum_k   (h_k  - \tr(M_k RR^T)) ^2 \notag
\\&- 
\sum_{(l,m)} \lambda^u_{(l,m)} (u_{(l,m)} - \tr(Y_{(l,m)}  RR^T)) 
  + \tfrac{\mu}{2} \sum_{(l,m)}   (u_{(l,m)}- \tr(Y_{(l,m)}  RR^T)) ^2 \notag
 \\&- 
\sum_{(l,m)} \lambda^v_{(l,m)} (v_{(l,m)} - \tr(\bar Y_{(l,m)}  RR^T)) 
  +  \tfrac{\mu}{2} \sum_{(l,m)}   (v_{(l,m)}- \tr(\bar Y_{(l,m)}  RR^T)) ^2 \notag
 \\&- 
\sum_{(l,m)} \lambda^z_{(l,m)} (z_{(l,m)} - u_{(l,m)}^2
-v^2_{(l,m)} ) 
 +  \tfrac{\mu}{2} \sum_{(l,m)}   (z_{(l,m)} - u_{(l,m)}^2
-v^2_{(l,m)}) ^2 + \nu \mathcal{R}. \notag
\end{align}
%}
Note that constants $\mu, \nu > 0$ pre-multiply regularisers, where $\mathcal{R}$ can often be 0 in practice,
although not in our analysis, where we require $\mathcal{R} = \det(R^T R)$,
which promotes low-rank solutions. 

%\todo[inline]{I think that $\mathcal{R}$ is not sufficiently explained...}

As suggested in Algorithm~\ref{alg:SCDM}, 
 we increase the rank $r = 1, 2, \ldots$ allowed in $W$ in an outer loop of the algorithm. 
In an inner loop, we use coordinate descent method to find an 
approximate minimizer of  
$\mathcal{L}(R,
t,h,g,u,v,z,\lambda^t, \lambda^g, \lambda^h,
\lambda^{u},\lambda^{v},\lambda^z)$, and denote the $k$-th iterate $\RK$.
%for fixed values of Lagrange multipliers
%$\lambda^\cdot$.
Note that variables $t,h,g,z$ have simple box constraints, which have to be considered outside of $\mathcal{L}$.

%\rr{
%In many cases, such a regularisation makes it possible to improve the covergence rates
%considerably, at the price of using a higher-order polynomial in the constraints.}

%After the approximate minimizer is found, Lagrange  multipliers $\lambda^\cdot$ and parameter $\mu$ are updated in a classical way.
 
%Note that the objective function as a function of $R$ is -- in general -- non-convex.  However, if we fix all values in $R$ except one, i.e. we would like to change only $(i,j)$-th element of $R$, it will turn out that one has to minimize 4 order polynomial.
%This polynomial has at most 3 real stationary points and can be found using a closed form formulas. The global minimum can be found by checking the objective value at each stationary point and choosing the best one. 

%Tohle muzeme prip. smazat.
%\rr{
%As we explain in Section \ref{sec:analysis},
%if the optimal solution $W^*$ of [LL], i.e. the SDP relaxation of \ref{R1} obtained by dropping the rank constraint,
%has a $\rank(W^*)=1$, we can find a vector $R^*$ such that $R^* (R^*)^T = W^*$, and hence a global solution of \eqref{PP4},
%without considering $r$ larger than 2.
%Otherwise, considering Observation $O_3$, the optimal solution $W^*$ of [LL] is not particularly useful.
%}
%Let us comment on the features of the proposed algorithm: % first, though.

In particular:

%In \cite{lavaei2012zero} authors provided a 
%a way how to construct optimal solution of \eqref{PP4} 
%if $\rank(W^*)\leq 2$ and under an assumption that some zero-duality-gap condition is satisfied (Theorem 2, ii) in \cite{lavaei2012zero}. However, this condition may not be satisfied for any problem ACOPT problem (however, in \cite{lavaei2012zero} they showed that the problem can be modified such that it would satisfy the zero-duality-gap condition).
 
 \begin{algorithm}[tb!]
 \begin{algorithmic}[1]
 \FOR{$r=1,2,\dots$}
 \STATE \label{alg:stp:initialpoint} choose $R \in \R^{m\times r}$ 
 \STATE compute corresponding values of $t,h,g,u,v,z$
 \STATE project $t,h,g,u,v,z$ onto the box constraints %to ensure feasibility
 \FOR{$k=0,1,2,\dots$} \label{inner-loop}
   \STATE 
   \textbf{in parallel}, minimize $\mathcal{L}$ in $t, g, h, u, v, z$, coordinate-wise   \label{updateT}
   \STATE 
   \textbf{in parallel}, minimize $\mathcal{L}$ in $R$, coordinate-wise   \label{updateR}
%to obtain an approximate solution of $\mathcal{L}$
%(we suggest that you update each coordinate only few times)
  \STATE update Lagrange multipliers $\lambda^t, \lambda^g, \lambda^h,
\lambda^{u},\lambda^{v},\lambda^z$ 
  \STATE update $\mu$
  \STATE terminate, if criteria are met
  \ENDFOR \label{inner-loop-end}
  \ENDFOR
 \end{algorithmic}
 \caption{A Low-Rank Coordinate Descent Algorithm} % Parallel
 \label{alg:SCDM}
\end{algorithm} 

\subsubsection*{The Outer Loop}

The outer loop (Lines 1-12) is known as the ``low-rank method'' \cite{BurerMonteiro2005}.
%It is known \cite{BurerMonteiro2005} that one has to perform at most $\mathcal{O}(\sqrt m)$ iterations of the outer loop, where $m$ is the number 
%of constraints in the problem. This follows from %Theorem *** %of Pataki \cite{pataki1998} and 
%Theorem 1.2 of Barvinok \cite{barvinok2001}.
%Later, \cite{BurerMonteiro2005} gave an elementary proof for a related Lemma 3.1. 
As suggested by Observation $O_3$, in the case of \ref{LL}, one may want to perform only two iterations $r = 1, 2$.
In the second iteration of the outer loop, one should like to test, 
whether the numerical rank of the iterate $R_k$ in the inner iteration $k$ has numerical rank $1$.
If this is the case, one can conclude the solution obtained for $r = 1$ is exact.
This test, sometimes known as ``flat extension'', has been studied both in terms of numerical implementations and applicability by Burer and Choi \cite{burer2006}.

\subsubsection*{The Inner Loop}

%For $r = 1$, 
The main computational expense of the proposed algorithm is to find an approximate minimum of 
$\mathcal{L}(R,
t,h,g,u,v,z,\lambda^t, \lambda^g, \lambda^h,
\lambda^{u},\lambda^{v},\lambda^z)$
with respect to $R,t,h,g,u,v,z$ within the inner loop (Lines 6--7).
Note that $\mathcal{L}$ as a function of $R$ is -- in general -- non-convex.  
The inner loop employs a simple iterative optimization strategy, known as 
the coordinate descent. There, 
two subsequent iterates differ only in a single block of coordinates.
In the very common special case, used here, we consider single coordinates,
in a cyclical fashion.
This algorithm has been used widely at least since 1950s.
%See \citep{Tseng:CCMCDM:Smooth, Tseng:CGDM:Nonsmooth, Tseng:CGDMLC:Nonsmooth, Tseng:CBCDM:Nonsmooth} for classical studies.
Recent theoretical guarantees of random coordinate descent algorithm are due to Nesterov \cite{Nesterov:2010RCDM} and the present authors \cite{RT2015,marecek2015distributed}.
See the survey of Wright \cite{Wright:ABCRRO} for more details.
%Coordinate descent algorithm is also closely related to 
%linear and non-linear Gauss-Seidel methods and related approaches \cite{Bertsekas-Book}.

\subsubsection*{The Closed-Form Step}

An important ingredient in the coordinate descent is a novel closed-form step.
Nevertheless, if we update only one scalar of $R$
at a time, and fix all other scalars, 
the minimisation problem turns out to be the minimisation of a fourth order polynomial.
In order to find the minimum of a polynomial $ax^4 + bx^3 + cx^2 + d x + 0$,
we need to find a real root of the polynomial $4ax^3 + 3b x^2 + 2c x  + d = 0$.
%for which there is a closed-form solution due to Cardano.
This polynomial has at most 3 real roots and can be found using closed form formulae due to Cardano \EDIT{ \cite{cardano1968ars}}. 
Whenever you fix the values across all coordinates except one, finding the best possible update for the one given coordinate 
requires either the minimisation of a quadratic convex function with respect to simple box constraints (for variables $t,g,h,z$) 
or minimisation of a polynomial function of degree 4
with no constraints (for variables $R, u, v$),
either of which can be done by checking the objective value at each out of 2 (for variables $t,g,h,z$) or 3 (for variables $R, u, v$)
stationary points and choosing the best one. 

\subsubsection*{The Parallelisation}

For instances large enough, one can easily exploit parallelism.
Notice that minimization of coordinates of $t,g,h,u,v,z$ can be carried out in parallel without any locks,
 as there are no dependences.
One can also update coordinates of $R$ in parallel, % \cite{bradley2011parallel}, 
 although some degradation of the speed-up thus obtainable is likely,
 as there can be some dependence in the updates. % Saha10finite, 
The degradation is hard to bound. Most analyses, c.f. \cite{Wright:ABCRRO}, %\citep{Tseng:CCMCDM:Smooth, Tseng:CGDM:Nonsmooth, Tseng:CGDMLC:Nonsmooth, Tseng:CBCDM:Nonsmooth, Yun2009, } 
hence focus on the uniformly random choice of (blocks of) coordinates,
although there are exceptions \cite{marecek2015distributed}.
Trivially, one could also parallelise the outer loop, for each $r$ that
should be considered.

%For some of the most recent analyses of asynchronous variants, see \cite{RT:hogwild,marecek2015distributed}.
%\rr{
%In every iteration of parallel coordinate descent, a subset of coordinate-blocks $\NS \subset \{ 1, 2, \ldots, n \}$ are chosen and for each of the chosen (blocks of) coordinates $i \in \NS$ is the update %computed and applied in parallel. The new iteration $x_{k+1}$ hence differs from the previous iteration $x_k$ in at most in $|\NS|$ coordinate-blocks.
% each computer chooses some (blocks of) coordinates, , in each iteration and computes the update for each .
%Notice that, if the step-direction were given by the choice of the blocks of coordinates,
% the step-length computation would need to take the parallelism into account.
%}

\subsubsection*{Sufficient Conditions for Termination of the Inner Loop}

%$\ $
%\rr{this should be rewritten, maybe say that "one could do that". I saw a benefit of this, when I put the solution obtained by our algorithm then Matpower finishes much faster}

For both our analysis and in our computational testing, we use a ``target infeasibility''
stopping criterion for the inner loop, considering squared error:
%{\small
\begin{align}
\EDIT{T_k(\RK)} = & \sum_k  (t_k  - \tr(Y_k \RK (\RK)^T))^2 + \label{eq:T}    
 \sum_k  (g_k -  \tr(\bar Y_k \RK (\RK)^T))^2 +  \\
 &  \sum_k  (h_k  - \tr(M_k \RK (\RK)^T))^2 + \notag    \sum_{(l,m)} (u_{(l,m)} - \tr(Y_{(l,m)}  \RK (\RK)^T))^2 + \notag \\
 & \sum_{(l,m)}   (v_{(l,m)} - \inTR(\bar Y_{(l,m)}  \RK (\RK)^T))^2 + \notag   \sum_{(l,m)}   (z_{(l,m)} - u_{(l,m)}^2 -v^2_{(l,m)} )^2. \notag 
\end{align}
%}
We choose the threshold to match the accuracy in terms of squared error obtained by Matpower using default settings on the same instance.
$\EDIT{T_k(\RK)} \le 0.00001$ is often sufficient.

%Notice that this makes use of the \emph{non-convex} Lagrangian of the
%polynomial optimisation problem, rather than the \emph{convex} augmented Lagrangian 
%of the semidefinite programming problem, where the $\log \det$ is difficult to work with.

\subsubsection*{The Initialisation}

In our analysis, we assume that the instance is feasible.
This is difficult to circumvent, considering Lehmann et al. \cite{Lehmann2015} have shown it is NP-Hard to test whether an instance of ACOPF is feasible.
In our numerical experiments, however, we choose $R$ such that each element is independently identically distributed, uniformly over [0, 1],
which need not be feasible for the instance of ACOPF. 
Subsequently, we compute $t,h,g,u,v,z$ to match the $R$, projecting the values onto the intervals given by the box-constraints. 
Although one may improve upon this simplistic choice by a variety of heuristics, it still performs well in practice.

\subsubsection*{The Choice of $\mu_k, \nu_k$}

The choice of $\mu_k, \nu_k$ may affect the performance of the algorithm.
In order to prove convergence, one requires $\lim_{k \to \infty} \nu_k \to 0$, % Page 12 of Burer2003
 but computationally, we consider $\mathcal{R} = 0$, which obliterates the need for varying $\nu_k$
%\rr{For $\mathcal{R} = 0$, one performs projections only onto the box-constraints.
%For $\mathcal{R} = \det(R^T R)$, which we consider in our analysis, one would have to 
and computation of the gradient,
%\begin{align*}
$\nabla \det(R^T R)$, %= R \; \textrm{cofactor}(R^T R),
%\end{align*}
as suggested by \cite{BurerMonteiro2005}. 

In order to prove convergence, any choice of $\mu_k > 0$ is sufficient.
Computationally, we use $\mu_k = \mu = 0.0001$ throughout much of the reported experiments.
On certain well-known pathological instances, c.f., Section~\ref{sec:pathological}, 
other choices may be sensible. Considering that many of those instances are very small,
the use of large step-sizes, such as $\mu_k = \mu = 0.01$ may seem justified. 
We have also experimented with an adaptive strategy for changing $\mu_k$, as detailed in Section~\ref{sec:adaptivemu}.
There, we require the infeasibility to shrink by a fixed factor between consecutive iterations. 
If such an infeasibility shrinkage is not feasible, one ignores the proposed iteration update and picks $\mu_{k+1} < \mu_{k}$ such
that the infeasibility shrinkage is observed.
Although the performance does improve slightly, % considering that this introduces another parameter, 
we have limited the use of this strategy to Section~\ref{sec:adaptivemu}.
We have also experimented with decreasing $\mu_k$ geometrically, i.e. $\mu_k = c \cdot \mu_{k-1}$,
but we have observed no additional benefits.
The fact one does not have to painstakingly tune the parameters is certainly a relief.

\section{Implementation Details}
\label{sec:details}

In order to understand the details of the implementation, it is important to realise that
\begin{itemize} 
\item the quadratic forms such as $Y_k xx^T$ in the polynomial optimisation problem \eqref{PP4} 
and $Y_k RR^T$ in the augmented Lagrangian \eqref{AugLagrangian}
are used for the simplicity of presentation. 
Indeed, the evaluation of the quadratic forms can be simplified considerably,
when one does not strive for the brevity of expression
\item \EDIT{many well-known benchmarks, including the Polish network, consider an extension of
the polynomial optimisation problem \eqref{PP4}.} \\[3mm]
\end{itemize}

\EDIT{
We elaborate upon these points in turn.

\subsection{The Simplifications}
}

Let us consider the example of $\trace(Y_k xx^T)$ in more detail. The evaluation of the 
quadratic form can be simplified to %4 look-ups into the vector $x$ and 
$16 \; \|Y_k\|_0$ multiplications,
%\todo[inline]{we have there also trace, so we do not need to compute the whole matrix, just diagonal, right? hence we need for every non-zero entry only 2 multiplications}
i.e., performed in time linear in the number $\|Y_k\|_0$ of non-zero elements of the matrix $Y_k$.
Moreover, the evaluation of the 
trace of the quadratic form requires only $2 \; \|Y_k\|_0$ multiplications,
as one needs to consider only the diagonal elements.
The number of non-zero elements varies in each power system, and is quadratic in the
number of buses for hypothetical systems, where there would be a branch between each pair of buses,
but the number of non-zero elements is linear in the number of buses, in practice.
Consequently, one can simplify the first step of the inner loop (Lines 6--7), where one minimises 
the augmented Lagrangian \eqref{AugLagrangian}, i.e., $\mathcal{L}(R,
t,h,g,u,v,z,\lambda^t, \lambda^g, \lambda^h,
\lambda^{u},\lambda^{v},\lambda^z)$, 
with respect to $t$, as follows. 
At iteration $k$, for each coordinate $j$, in parallel, one computes the update:
%\todo[inline]{I think we should put some trace in the last term in numerator in next equations???}
\begin{align}
t_j^{(k)} := & - \frac{ \pi^t_j + c^1_j S_b + 2 c^2_j S_b^2 (P_j^g)^2  - \frac{\trace(Y_k \RK(\RK)^T)}{\mu}}{ \frac{1}{\mu} + 2 c^2_j S_b^2 },
\end{align}
where $S_b$ is the base power of the per-unit system, such as 100 MVA, and $\pi^t$ are the residuals:
%\todo[inline]{I think we should put some trace in the last term in numerator in next equations???}
\begin{align}
\pi_j^t := & - \frac{ \trace(Y_k R^{(k-1)}(R^{(k-1)})^T) - t_j^{(k-1)} }{\mu}. \label{residual-pit}
\end{align}
Notice that the terms $c^1_j S_b$ and $2 c^2_j S_b^2$ are constant throughout the run and can hence be pre-computed,
while the residuals \eqref{residual-pit} have to be precomputed only once per iteration, e.g., 
just before the termination criteria are evaluated (Line 10), and hence there are only 3 multiplications and 1 division involved,
in addition to the run-time of the evaluation of the quadratic form.
Subsequently, one projects onto the box-constraints, i.e., if $t_j$ is large than the upper bound, it is set to the upper bound.
If $t_j$ is smaller than the lower bound, it is set to the lower bound.
Similar simplifications can be made for minimisation of the augmented Lagrangian \eqref{AugLagrangian}
with respect to other variables.\\[3mm]

%Notice also that the measure of infeasibility \eqref{eq:T} can be 
%maintained using residuals such as $pi^t$ above \eqref{residual-pit}, 
%which are also useful in the updates.

\EDIT{
\subsection{The Extensions}
}

\EDIT{In an often considered, but rarely spelled out} \cite{molzahn2011} extension of the ACOPF problem \eqref{eqn:pp4}, 
\EDIT{one allows} for tap-changing and phase-shifting transformers, as well as parallel lines
and multiple generators connected to one bus.
\EDIT{Let us denote the 
 the total line charging susceptance (p.u.) by $b_{lm}$,
 the transformer off nominal turns ratio by $t_{lm}$,
and the transformer phase shift angle by $\phi_{lm}$.
Then, the thermal limits \eqref{eqn:pp4} become:}

\begin{align}
\left[ {\begin{array}{*{20}c}
   {(S_{lm}^{\max})^2 } & {-\trace(Z_{lm} xx^T)}  & {-\trace(\bar{Z}_{lm} xx^T)}  \\
   {-\trace( Z_{lm} xx^T)} & {1}  & {0}  \\
   {-\trace( \bar{Z}_{lm} xx^T)} & {0}  & {1}  \\
\end{array}} \right] & \succeq 0  \\
\left[ {\begin{array}{*{20}c}
   { (S_{lm}^{\max})^2 } & {-\trace(\Upsilon_{lm} xx^T)}  & {-\trace(\bar{\Upsilon}_{lm} xx^T)}  \\
   {-\trace( \Upsilon_{lm} xx^T)} & {1}  & {0}  \\
   {-\trace( \bar{\Upsilon}_{lm} xx^T)} & {0}  & {1}  \\
\end{array}} \right] & \succeq 0 
\end{align}

where:
%listed in Figure~\ref{definitions}.
% WHAT IS IN MATLAB _km is now Upsilon
%\begin{figure}[thb!]
\begin{align}
Z_{lm} &= \frac{g_{lm}}{t_{lm}^2} (e_l e_l^T+e_{l+|N|} e_{l+|N|}^T) \label{ZlmLong} \\
       & - \left(\frac{g_{lm} \cos(\phi_{lm})+b_{lm} \cos(\phi_{lm}+\frac{\pi}{2})}{2 t_{lm}}\right) (e_l e_{m}^T+e_{m} e_l^T+e_{l+|N|} e_{m+|N|}^T+e_{m+|N|} e_{l+|N|}^T) \notag \\
       & + \left(\frac{g_{lm} \sin(\phi_{lm})+b_{lm} \sin(\phi_{lm}+\frac{\pi}{2})}{2 t_{lm}}\right) (e_l e_{m+|N|}^T+e_{m+|N|} e_l^T-e_{l+|N|} e_{m}^T-e_{m} e_{l+|N|}^T) \notag \\
\Upsilon_{lm} &=  g_{lm} (e_{m} e_{m}^T+e_{m+|N|} e_{m+|N|}^T) \\
       & - \left(\frac{g_{lm} \cos(-\phi_{lm})+b_{lm} \cos(-\phi_{lm}+\frac{\pi}{2})}{2 t_{lm}}\right) (e_l e_{m}^T+e_{m} e_l^T+e_{l+|N|} e_{m+|N|}^T+e_{m+|N|} e_{l+|N|}^T) \notag \\
       & + \left(\frac{g_{lm} \sin(-\phi_{lm})+b_{lm} \sin(-\phi_{lm}+\frac{\pi}{2})}{2 t_{lm}}\right) (e_{l+|N|} e_{m}^T+e_{m} e_{l+|N|}^T-e_l e_{m+|N|}^T-e_{m+|N|} e_l^T) \notag \\
\bar{Z}_{lm} &= -\frac{2 b_{lm}+\bar{b}_{lm}}{2 t_{lm}^2} (e_l e_l^T+e_{l+|N|} e_{l+|N|}^T) \\
       & + \left(\frac{g_{lm} \cos(\phi_{lm})+b_{lm} \cos(\phi_{lm}+\frac{\pi}{2})}{2 t_{lm}}\right) (e_l e_{m+|N|}^T+e_{m+|N|} e_l^T-e_{l+|N|} e_{m}^T-e_{m} e_{l+|N|}^T) \notag \\
       & + \left(\frac{g_{lm} \sin(\phi_{lm})+b_{lm} \sin(\phi_{lm}+\frac{\pi}{2})}{2 t_{lm}}\right) (e_l e_{m}^T+e_{m} e_l^T+e_{l+|N|} e_{m+|N|}^T+e_{m+|N|} e_{l+|N|}^T) \notag \\
\bar{\Upsilon}_{lm} &= -\frac{2 b_{lm}+\bar{b}_{lm}}{2} (e_{m} e_{m}^T+e_{m+|N|} e_{m+|N|}^T) \\
      & + \left(\frac{g_{lm} \cos(-\phi_{lm})+b_{lm} \cos(-\phi_{lm}+\frac{\pi}{2})}{2 t_{lm}}\right) (e_{l+|N|} e_{m}^T+e_{m} e_{l+|N|}^T-e_l e_{m+|N|}^T-e_{m+|N|} e_l^T) \notag \\
      & + \left(\frac{g_{lm} \sin(-\phi_{lm})+b_{lm} \sin(-\phi_{lm}+\frac{\pi}{2})}{2 t_{lm}}\right) (e_l e_{m}^T+e_{m} e_l^T+e_{l+|N|} e_{m+|N|}^T+e_{m+|N|} e_{l+|N|}^T) \notag 
\end{align}
%\caption{Definitions of $Z_{lm}, \Upsilon_{lm}, \bar{Z}_{lm}$, and $\bar{\Upsilon}_{lm}$.}
%\label{definitions}
%\end{figure}

where $e_k$ be the $k^{th}$ standard basis vector in $\mathbb{R}^{n}$, 
one can perform similar simplifications.

For instance, following the pre-computation of vectors $c^A, c^C, c^D \in \R^{|M|}$ prior to the outer loop of the algorithm,
the evaluation of the trace of the quadratic form, $\trace(Z_{lm} xx^T)$,
%\todo[inline]{Should we have also trace there?}
 considering \eqref{ZlmLong} can be implemented using 
  4 look-ups into the vector $x$ and 13 float-float multiplications:
%  \todo[inline]{Should we have also trace there on LHS?}
\begin{align}
\trace(Z_{lm} xx^T) = & \; c^A_{lm} (x_l^2 + x_{l+|N|}^2) \\
                      & +  c^C_{lm} (2 x_l x_m + 2 x_{l+|N|} x_{m+|N|}) \notag \\
                      &  +  c^D_{lm} (2 x_l x_{m + |N|} - 2 x_{m} x_{l+|N|}). \notag 
\end{align}
One can simplify the multiplication in a similar fashion for the remaining expressions involving 
(traces of) quadratic forms of $\Upsilon_{lm}, \bar{Z}_{lm}$, and $\bar{\Upsilon}_{lm}$ as well.\\[3mm]

\section{An Analysis}
\label{sec:analysis}

The analysis needs to distinguish between optima of 
 the semidefinite programming problem \ref{LL}, which is convex, 
and stationary points, local optima, and global optima of \ref{RrBC}, which is the non-convex rank-constrained problem.
%Specifically, for \eqref{RrBC} we are not guaranteed 
% that local optima are global optima.
%Let us clarify the distinction between stationary points and local optima.
%Clearly, for the \emph{convex} semidefinite programming problems, 
%stationary points are local optima which are global optima.
%For the \emph{non-convex} problems such as [ACOPF], [PP4], \ref{RrBC},
%the three can be distinct. 
Let us illustrate this difference with an example:

\begin{example}%[Example of SDP where a stationary point is not a local optimum]
Consider the following simple rank-constrained problem:
\begin{align}
\label{exrank}
\max_{W \in \mathcal{S}_+^2, \rank(W)=1}
  \tr(\diag(3,1) W), 
\end{align}
subject to $\tr(I W ) =1$ and
\begin{align}
W^* & := \begin{pmatrix}
1&0\\0&0\end{pmatrix} = (1,0) (1,0)^T, \notag \\
\tilde W & := 
\begin{pmatrix}
0&0\\0&1\end{pmatrix} = (0,1) (0,1)^T =: \tilde R \tilde R^T. \notag 
\end{align}
$W^*$ is the unique optimal solution of \eqref{exrank}, as well as the optimal solution of the SDP relaxation, where the rank constraint is dropped.
There exists a Lagrange multiplier such that $\tilde W$ is a stationary
point, though. % $\lambda$ 
Let us define a Lagrange function $\mathcal{L}(R, \lambda)
 = 
\tr(\diag(3,1) RR^T)
 + \lambda (\tr(I RR^T ) - 1)$.
 Then 
 $\nabla_R \mathcal{L}( RR^T,\lambda)
  = 2\diag(3,1) R +2 \lambda R  $.
  If we plug in $\tilde R$ and $\tilde\lambda = -1$
  we obtain
 $\nabla_R \mathcal{L}( \tilde R\tilde R^T,\tilde \lambda)
  =
   2\diag(3,1) (0,1)^T +2 \tilde\lambda (0,1)^T
  = 
   (0,0)^T  
    $.  
  Hence $(\tilde R, \tilde \lambda)$ is a stationary point of a Lagrange function. However, one can follow the proof of Proposition \ref{prop:3}
  to show that $[\tilde R, 0]$ is not a local optimum solution of SDP relaxation.
  \qed
\end{example}

Indeed, in generic non-convex quadratic programming, 
the test whether a stationary point solution is a local optimum \cite{murty1987} % ,Pardalos1988
is NP-Hard. %as second-order necessary conditions for optimality need not be sufficient.
For ACOPF, there are a number of sufficient conditions known, e.g. 
%the primal-dual condition of Molzahn et al. 
\cite{MolzahnDeMarco2014}.
Under strong assumptions, motivated by Observation $O_3$, we provide 
necessary and sufficient conditions
based on those of Grippo et al.\ \cite{grippo2009necessary}.
 %the work of %Grippo et al.
 % grippo2011unconstrained
%In keeping with their notation, we summarise [LL] as a semidefinite programming 
%problem in the standard form
%Further, 
We use $\otimes$ for Kronecker product and 
$I_r$ for identity matrix in $\R^{r \times r}$.
%and our reformulation [R$k$BC].
%We use $\sigma(A) = \{ \lambda_1, \ldots, \lambda_n \}$ to denote the
%spectrum of a matrix $A \in \mathcal{S}^n$. 

\begin{proposition}
Consider the SDP relaxation obtained from \ref{RrBC} by dropping the rank constraint and writing it down as 
$\min \tr(QW)$ such that $\tr(A_i W) = b_i$ for constraints $i = 1, \ldots, m$
in variable $W \succeq 0, W \in \mathcal{S}^n$.
If there exists an optimum $W^*$ with rank $r$ for the SDP relaxation,
for any point $R \in \R^{n \times r}$, $RR^T$ is a global minimiser of \ref{RrBC} if and only if there exists a $\lambda^* \in \R^m$ such that:
\begin{align}
\left[ {\left( {Q + \sum\limits_{i = 1}^m {\lambda^*_i A_i} } \right) \otimes I_r} \right] R & = 0 & \\
Q + \sum\limits_{i=1}^{m} \lambda_{i}^{*} A_i & \succeq 0 & \notag \\
R^{T}(A_i \otimes I_r) R & = b_i & \forall i = 1, \ldots, m. \notag
\end{align}
\end{proposition}

\begin{proof}
The proof is based on Proposition 3 of Grippo et al.\ \cite{grippo2009necessary},
which requires the existence of rank $r$ optimum and strong duality of the SDP relaxation.
The existence of an optimum of \ref{LL} with rank $r$ is assumed.
One can use Theorem 1 \cite{Josz2015} and Theorem 1 of \cite{Ghaddar2015} to show 
that a ball constraint is sufficient for strong duality in the SDP relaxation, c.f. Observation $O_4$.
\end{proof}

Notice that the test for 
whether there exists an optimum $W^*$ with rank $r$ for the SDP relaxation
is suggested Proposition 5, i.e. by solving for rank $r+1$.

%notice that we focus on an idealised version of 
%Algorithm~\ref{alg:SCDM} with $\mathcal{R} = \det(R^T R)$.

%\begin{proposition}
%Let $\mu>0$
%and $W^{(k)}$ be the (random) matrices produced by Algorithm~\ref{alg:SCDM} after $k$ iterations.
%Then Algorithm~\ref{alg:SCDM} is monotonic, i.e., for all $k\geq 0$,
%\begin{equation}\label{eq:monotonicity}
%  0 \leq  f(W^{(k+1)}) \leq f(W^{(k)}).
%\end{equation}
%Moreover, almost surely,
%\begin{equation} 
%\nabla_X f(W^{(k)})\to 0. 
%\end{equation}
%\todo[inline]{This could we true if we have constrained optimization, I need to think about this twice:)}
%\end{proposition}
%
%\begin{proof}[Sketch of the proof]
%Monotonicity can be deduced from ... Then assumption that $\mu>0$ together with monotonicity \eqref{eq:monotonicity}
%implies that the levelset ... is bounded. Hence, the Lipschitz constants $W$ and $V$ are bounded above. The rest follows again from
%...
%\end{proof}
%
%This can be had for any particular choice of the rank $r$. Notice, however, 
%that for some choices of rank $r$, we can prove a much stronger statement:

Next, let us consider the convergence:

\begin{proposition}
\label{prop:2}
%For every instance, 
There exists an instance-specific constant $r'$, such that for every $r \ge r'$,
whenever 
Algorithm~\ref{alg:SCDM} with $\EDIT{\{ \RK \} \in \R^{n \times r}}, \mathcal{R} = \det(R^T R)$ and $\lim_{k \to \infty} \nu_k \to 0$ produces solution with $\lim_{k \to \infty} \EDIT{T_k(\RK)} \to 0$
and a local optimum % stationary point ?!!
$\RK$ is generated within the inner loop (5--11),
%i.e. $\nabla_X f(W^{(k)})\to 0$, we have
$\RK (\RK)^T$ is an optimal solution to \ref{LL}. %with cost $\min([LL])$.
% where 
%[LL] is the semidefinite programming relaxation obtained by replacing the rank-constraint in [R1] with a 
%constraint to the positive-semidefinite cone, and 
%[OP$_4$-H$_1$] is the semidefinite programming relaxation of Ghaddar et al.
% obtained by replacing the rank-constraint in [R1] with a 
%constraint to the positive-semidefinite cone.
Moreover, $r'$ depends on the number of constraints $m$ in the optimisation problem and is $\mathcal{O}(\sqrt m)$.
\end{proposition}

\begin{proof}
The proof follows from Theorem 3.3 of Burer and Monteiro \cite{BurerMonteiro2005}.
One can rephrase Theorem 3.3 to show that if $\{ \RK \} \in \R^{n \times r}$ is a bounded sequence 
such that: 
\begin{itemize}
\item[$C_1$:] $\lim_{k \to \infty} T_k = 0$
\item[$C_2$:]  $\lim_{k \to \infty} \nabla \mathcal{L}(\RK) = 0$ 
\item[$C_3$:]  $\lim \inf_{k \to \infty} \nabla^2 \mathcal{L} (\RK)(H^k, H^k) \ge 0$ for all bounded sequences $\{ H^k \}$, $H^k \in \R^{n \times r}$
% TODO: We never explain the (H^k, H^k) notation, but it is not the same as:
%\item[$C_3$:]  $\lim \inf_{k \to \infty} \nabla^2 \mathcal{L} (\RK) (H^k (H^k)^T) \ge 0$ for all bounded sequences $\{ H^k \}$, $H^k \in \R^{n \times r}$ ?!
\item[$C_4$:]  $\rank(\RK) < r$ for all $k$
\end{itemize}
every accumulation point of $\RK (\RK)^T$ is an optimal solution of \ref{LL}.
Let us show that these four conditions are satisfied, in turn.
Condition $C_1$, which effectively says that $W=\RK (\RK)^T$ should be feasible with respect to constraints (\ref{R1-Constraint1}--\ref{R1-Constraint4}), 
is affected by the termination criteria of the inner loop, albeit only approximately for a finite $k$ and finite machine precision.
Conditions $C_2$ and $C_3$ follow from our assumption that $\RK$ is a local optimum. %stationary point.
%\todo[inline]{Then we need to check: 
The satisfaction of Condition $C_4$ can be shown in two steps:
First, there exists an $r'$, such that for every feasible semidefinite programming problem with $m$ constraints, 
there exists an optimal solution with a rank bounded from above by $r'$.
This $r'$ is $\mathcal{O}(\sqrt m)$.
This follows from Theorem 1.2 of Barvinok \cite{barvinok2001},
as explained by Pataki \cite{pataki1998}.
Second, the $\mathcal{R} = \det(R^T R)$ regularisation 
forces the lower-rank optimum to be chosen, should 
there be multiple optima with different ranks.
This can be seen easily by contradiction.
%or its simplification in Lemma 3.1 of \cite{BurerMonteiro2005}.
%The $\mathcal{O}(\sqrt m)$ bound on $r'$ follows from Theorem 1.2 of Barvinok \cite{barvinok2001}
%Later, \cite{BurerMonteiro2005} gave an elementary proof for a related Lemma 3.1. 
Finally, one can remove the requirement on the sequence to be bounded by the arguments of Burer and Monteiro \cite{BurerMonteiro2005}, as per Theorem 5.4.
%One could also weaken the assumption we obtain a local optimum by using Theorem 5.3 of \cite{BurerMonteiro2005}.
\end{proof}

%Question: Do the box-constraints count into $m$? If it were possible to extend this such
%that they would not ... 
%\todo[inline]{maybe not, let me check that theorem in \cite{BurerMonteiro2005}}
%Consequently, we obtain the following iteration complexity to global optima:

%Next, we clearly need means of distinguishing a stationary point and a local optimum:

Further,

\begin{proposition}
\label{prop:3}
Consider an instance of ACOPF such that there exists an optimum solution $W^*$ of \ref{LL} with rank 1 and $\forall k: c^2_k\geq 0$.
Let $\RK$ 
be a iterate produced by the  Algorithm~\ref{alg:SCDM}
when run with $r=1$
and moreover it is such that 
$T_k = 0$.
Then if $R R^T$ (where $R = [ \RK, {\bf 0}]$) is a local optimum 
of [R2BC], then 
$\RK (\RK)^T$ is a global optimum solution of [R1BC] and \ref{PP4}.
\end{proposition}
\begin{proof}
We will prove this proposition by contradiction.
For the sake of contradiction, we assume that
$R$ is a local optimum and 
$\RK$ is not a global optimum.
Therefore, we know that objective function of 
[R1BC] for $W_1 = W^*$ is smaller than 
for $W_2 = \RK (\RK)^T$ and 
both $W_1$ and $W_2$ are feasible.
By the assumption on optimum solution $W^*$ of \ref{LL}, we know that
$W^*$ can be written as
$W^*= w w^T$.
Now, it is easy to observe
that 
if we define
$\tilde R^\lambda=[(1-\sqrt{\lambda}) \RK, \sqrt{\lambda} w^*]$
then 
for any $\lambda\in [0,1]$ this vector is feasible for 
[R2BC].
Moreover, for $\lambda = 0$ we have $\tilde R^0 = R$.
Because the objective function $F$ of \ref{RrBC} is convex in $W$ with $c_k^2 \geq 0$,
we have that
$$
F( \tilde R^\lambda (\tilde R^\lambda)^T )
\leq
(1-\lambda) F(W_2) + \lambda F(W_1)
<  F(W_2)
$$
and hence for all $\lambda \in (0,1]$ we have
that
$F( \tilde R^\lambda (\tilde R^\lambda)^T )
 < F(W_2)$,
which is a contradiction with the assumption that
$R$ is a local optimum.
%\todo[inline]{Note than Monteiro 2011 needs some kind of regularity assumption which is not needed here as we have a special structure}
\end{proof}

%\todo[inline]{try to prove the thing about stationary of the local solution}

Overall, 

\begin{remark}
The first iteration of the outer loop of Algorithm~\ref{alg:SCDM} 
%with $\mathcal{R} = \det(R^T R)$ 
may produce a global optimum to \ref{R1} and \ref{RrBC} and \ref{PP4}, as suggested in Propositions \ref{prop:1}--\ref{prop:3}. 
The second and subsequent iterations of the outer loop of Algorithm~\ref{alg:SCDM} 
may find the optimum of \ref{LL}, the semidefinite programming problem, as suggested in Proposition \ref{prop:2},
but for $\mathcal{R} = 0$, they are not guaranteed to find 
the global optimum of \ref{R1} nor \ref{RrBC} nor \ref{PP4}. 
\end{remark}

One should also like contrast the ability to extract low-rank solutions
with other methods:

\begin{remark}
Whenever there are two or more optima of \ref{LL} with two or more distinct ranks, the maximum rank solutions are in the relative interior \EDIT{of the optimum face} of the feasible set, as per Lemma 1.4 in \cite{laurent2009sums}.
Primal-dual interior-point algorithms for semidefinite programming,
such as SeDuMi \cite{sturm1999},
in such a case return a solution with maximum rank.
\end{remark}

Put bluntly, one may conclude that as long as one seeks the exact optimum of \ref{ACOPF}, 
it does not make sense to perform more than two iterations of the outer loop of Algorithm~\ref{alg:SCDM}.
When it becomes clear by the second iteration that the rank-one optimum of \ref{LL} has not been extracted,
one should like to consider stronger convexifications \cite{Ghaddar2015}.
Notice that this advice is independent of whether one uses $\mathcal{R} = \det(R^T R)$ or $0$.
Although $\mathcal{R} = 0$ or $\nu = 0$ does not guarantee the recovery of low-rank solutions of \ref{LL}, 
\EDIT{in some cases \cite{bandeira2016low,boumal2016non},}
$\mathcal{R} = \det(R^T R)$ and $\nu > 0$ does not guarantee that the low-rank solution of the augmented Lagrangian \eqref{AugLagrangian} 
coincides with the optimum of \ref{ACOPF}, in some other cases.
It may hence be preferable to use $\mathcal{R} = 0$, as it allows \EDIT{for more iterations per second} %of Algorithm~\ref{alg:SCDM}
and post hoc testing of global optimality.

%This is explained by Observation $O_3$.

\section{Numerical Experiments}
\label{sec:numerical}

We have implemented Algorithm~\ref{alg:SCDM} in C++ with GSL and OpenMP and tested it on a collection \cite{Matpower} of well-known instances. % two instances from Bukhsh et al. \cite{Bukhsh2013}
%and two from Ghaddar et al. \cite{Ghaddar2015}.
For comparison, we have used three solvers specialised to ACOPF:
\begin{itemize}
\item the MATLAB-based Matpower Interior Point Solver (MIPS) version 1.2 (dated March 20th, 2015), which has been developed by Zimmerman et al.\ \cite{Matpower}
\item the C-based Semidefinite Programming optimal Power Flow ({\tt sdp\_pf}) version 1.0 (dated January 17th, 2014), 
%which is the Lavaei-Low SDP 
which has been developed by Mohlzahn et al.\ \cite{molzahn2011} using SeDuMi of Sturm et al.\ \cite{sturm1999}
%and shipped with Matpower 5.1 using 
\item the MEX-based {\tt OPF\_Solver} beta version dated December 13th, 2014, listed as the most current as of June 1st, 2016\footnote{\url{http://ieor.berkeley.edu/~lavaei/Software.html}}, which has been developed by Lavaei et al.\ \cite{7065336} using and SDPT3 of T{\"u}t{\"u}nc{\"u} et al.\ \cite{toh1999sdpt3,tutuncu2003solving}.
Notice that {\tt OPF\_Solver} produces 
\EDIT{feasible points with objective values that are near the values of the global optima}  across all instances tested, for some, non-default settings;
we have used the per-instance settings of epB, epL, and line\_prob, as suggested by the authors.
\end{itemize}
\EDIT{
For the comparison presented in Table~\ref{tab1}, we have used a standard laptop with Intel i5-2520M processor and 4 GB of RAM. 
We believe this is fair, as the solvers we compare with cannot make a good use of a more powerful machine.
%For the comparison with general-purpose interior-point methods, we have used a Lehigh University cluster, ***.
For the presentation of scalability of our code, we have used a 
machine with 24 Intel E5-2620v3 clocked at 2.40GHz and 128GB of RAM, but used only the numbers of cores listed.

We have also tested six general-purpose semidefinite-programming (SDP) solvers:
\begin{itemize}
\item CSDP version 6.0.1, which has been developed by Borchers \cite{borchers1999csdp}
\item MOSEK version 7.0, which has been developed by MOSEK ApS \cite{Andersen2003,permenter2015solving}
\item SeDuMi version 1.32, which has been developed by Sturm et al.\ \cite{sturm1999}
\item SDPA version 7.0, which has been developed by the SDPA group \cite{yamashita2012latest}
\item SDPT3 version 4.0, which has been developed by T{\"u}t{\"u}nc{\"u} et al.\ \cite{toh1999sdpt3,tutuncu2003solving}
\item SDPLR version 1.03 (beta), which has been developed by Burer and Monteiro \cite{BurerMonteiro2003}
%\item PenSDP version 2.2, which has been developed by Ko{\v c}vara and Stingl \cite{kocvara2006pensdp}.
\end{itemize}
Five of the codes (CSDP, SeDuMi, SDPA, SDPLR, and SDPT3) have been tested at the NEOS 7 facility \cite{714603} at the University of Wisconsin in Madison,
 where there are two Intel Xeon E5-2698 processors clocked at 2.3GHz and 192 GB of RAM per node.  
MOSEK has been used at a cluster equipped with one AMD Opteron 6128 processor clocked with 4 cores at 2.0 GHz and 32 GB of RAM per node,
as per the license to Lehigh University.
}

%in line with the experiments  \cite{BurerMonteiro2005}, 
 
\subsection{IEEE Test Cases}

Our main focus has been on the IEEE test cases. 
In Table~\ref{tab1}, we compare the run-time 
of \EDIT{our implementation of Algorithm~\ref{alg:SCDM}} with the run-time of the three leading solvers for the ACOPF listed above, two of which ({\tt sdp\_pf}, {\tt OPF\_Solver}) use elaborate tree-width decompositions.
In order to obtain the numbers, we ran Matpower first using default settings,
record the accuracy with respect to squared error $T_k$ \eqref{eq:T}, ran our solver up to the same accuracy,
and record the time and the objective function.
%Specifically, we should like to point to the instance
%case9mod of \cite{Bukhsh2013}, where the objective function value is reduced by 
%38.19\%.
%and to the standard case30Q,
%where the objective function value is reduced by 7.97\%.

\EDIT{
In Table~\ref{tabComparisonWithGP}, 
we present the run-time of six popular general-purpose SDP solvers for comparison.
We note that we have used the default parameters and tolerances of each code, so the precision may no longer match Matpower.
We list the reported CPU time rounded to one decimal digit, if that yields a non-zero number, and to one significant digit, otherwise.  
When we display a dash, no feasible solution has been found; the severity of the constraint violation and abruptness of the termination of the solver vary widely. 
For example, MOSEK terminates very abruptly with fatal error stopenv on five of the instances, while SDPT3 often runs into numerical difficulties. %  while reporting MSK\_RES\_TRM\_STALL with moderate primal constraint violations in the remainder of cases
Outside of SDPLR on the largest instance, no solver ran out of memory, iteration limit, or time limit.
In this comparison, SeDuMi seems to be most robust solver. 
SDPLR is also rather robust, but three orders of magnitude slower. 
Throughout, all interior-point methods (CSDP, Mosek, SeDuMi, SDPA, and SDPT3) perform much better on the dual of the SDP, than on the primal.
Please note that direct comparison with the run-time of our implementation of Algorithm~\ref{alg:SCDM}
reported in Table~\ref{tab1} is no possible, due to the use of three different platforms:
Five of the codes (CSDP, SeDuMi, SDPA, SDPLR, and SDPT3) have been tested at the NEOS facility, which does not allow for our code to be run, 
while Mosek has been tested at a Lehigh University facility.
Still, either facility has machines considerably more powerful than the machine used in the tests above, and we report the CPU time reported
by the individual solvers, rather than the wall-clock time, so we believe it is fair to claim our specialised solver outperforms the general-purpose solvers.
}

% Overall, 
%Our code seems to fast, whenever all three solvers produce a solution. 
%Our code also produces solutions for many instances, where
%{\tt sdp\_pf} does not.

%In Table~\ref{tab1}, we provide a comparison of solution
%quality, where we compare the best solutions obtained using our approach
%(often past the accuracy-based stopping criterion of the previous paragraph was reached)
%against the global optima of Ghaddar et al. \cite{Ghaddar2015}.

\EDIT{
\subsection{NESTA Test Cases}

Next, we have tested our approach on the recently introduced test cases from the NICTA Energy System Test Case Archive (NESTA) \cite{NESTA}.
There, bounds on commonly used IEEE test cases have been carefully tightened to make even the search for a feasible solution difficult for  many solvers.
For example, in the so called Active Power Increase (API) test cases,
the active power demands have been increased proportionally
throughout  the  network  so as to make thermal  limits  active.
In Table~\ref{tab2}, we present the results.
Out of the 11 API instances tested, Matpower and {\tt sdp\_pf} fail to find feasible solutions for three instances each.
Our implementation of Algorithm~\ref{alg:SCDM},
using default settings, including $\mathcal{R} = 0$, 
obtains feasible solutions across all the 11 instances, while improving over the objective function values obtained by either Matpower or {\tt sdp\_pf}.
For example, on the instance nesta\_case30\_as\_\_api, we improve the objective function value by about 10\%,
while on the instance nesta\_case30\_fsr\_\_api, we improve the objective by considerably more than 10\%.
Although we have not tested all instances of the NESTA archive, we are very happy with these preliminary results.
}

\subsection{Pathological Instances}
\label{sec:pathological}

Additionally, we have tested our approach on a number 
of recently introduced pathological instances \cite{Lesieutre2011,molzahn2013application,Bukhsh2013,Ghaddar2015,kocuk2016inexactness}. 
Depending on the choices of $R, \nu,$ and $\mu$, Algorithm~\ref{alg:SCDM} may perform better or worse than the relaxation of Lavaei and Low \eqref{LL}.
Let us illustrate this on the suggested settings of 
$\mathcal{R} = 0$, $\nu$ arbitrary, and single-thread execution.
In the example of Bukhsh et al. \cite{Bukhsh2013}, known as case2w, there are two local optima. Whereas many heuristics may fail or find the local optimum with cost 905.73, we are able to find the exact optimum with cost 877.78
in 0.49 seconds up to the infeasibility of $9.38 \cdot 10^{-6}$
with $\mu = 0.01$
and %as 877.06 
up to infeasibility of $7.23 \cdot 10^{-6}$ in 0.94 seconds
with $\mu = 0.0001$.
When we replace $V_2^{\max} = 1.05$ with 1.022, as in \cite{Ghaddar2015},
the instance becomes harder still, and the optimum of the relaxation of Lavaei and Low \eqref{LL} \EDIT{
is not rank-1. (Although one could project onto the feasible set of [R1BC] and extract a feasible solution of \ref{PP4}, it would not be the optimum of \ref{PP4}.)
}
%In contrast,  who also found the exact optimum with cost 905.73. 
With $\mu=0.01$, we find only a local optimum with cost 888.05 up to infeasibility of $8.29 \cdot 10^{-6}$, but with 
$\mu = 0.0001$, we do find the local optimum,
which in our evaluation has cost 905.66
up to infeasibility of $1.46 \cdot 10^{-9}$.
We stress that this optimum is not the solution of plain \eqref{LL}, but at the same time that, we do not provide
any guarantees of improving upon 
the relaxation of Lavaei and Low \eqref{LL}. 
In the example case9mod of  Bukhsh et al. \cite{Bukhsh2013},
$\mu=0.01$ does not converge within 10000 iterations, but
using $\mu = 0.0001$, we are able to 
find the exact optimum of 3087.84
and infeasibility $9.43 \cdot 10^{-12}$ in 12.16 seconds.
In the example case39mod2 of  Bukhsh et al. \cite{Bukhsh2013},
we are able to find only a local optimum with cost 944.71 and infeasibility $5.63 \cdot 10^{-8}$ after 45.03 seconds,
whereas the present-best known solution has cost 941.74.
In the example of Molzahn et al. 
\cite{Lesieutre2011,molzahn2013application}, which is known as
LMBM3, we find a solution with cost 5688.10 
up to infeasibility of $9.98 \cdot 10^{-6}$ in 0.88 seconds 
using $\mu = 0.01$ 
and a solution with cost 5694.34 
up to infeasibility of $1.14 \cdot 10^{-6}$ in 0.56
using $\mu = 0.0001$.
This illustrates that the choice of $\mu$ is important and
the default value, albeit suitable for many instances,
is not the best one, universally.

%\TODO{TODO: Rerun the instances with -P }

%\todo[inline]{should we remove the MOD instances? I guess that  whose, 
%where we produce the same solution as matpower are ok?}

\begin{sidewaystable}
\caption{The results of our numerical experiments on IEEE test cases and certain well-known pathological instances. Dash (--) indicates no feasible solution has been provided, often due to numerical issues or the lack of convergence. }
\label{tab1}
\begin{center}

\begin{tabular}{l|r|r|r|r|r|r|r|r|r} 
   \multicolumn{2}{c|}{Instance}
 & \multicolumn{2}{c|}{MATPOWER}
 & \multicolumn{2}{c|}{{\tt sdp\_pf}}  
 & \multicolumn{2}{c|}{{\tt OPF\_Solver}}  
 & \multicolumn{2}{c}{Algorithm~\ref{alg:SCDM}}  
 \\ \hline 
Name & Ref. & Obj. & Time [s] &  
        Obj. &  Time [s]
  &
    Obj. & Time [s]
  &
    Obj. & Time [s]

\\ 
\hline

case2w
& \cite{Bukhsh2013}
& ---
& ---
& 877.78
& 17.08 
& 877.78% really 7778
& 2.52
&  877.78    & 0.077

\\
case3w
& \cite{Bukhsh2013}
& ---
& ---
& 560.53
& 0.56
& 560.53
& 2.70
& 560.53     & 0.166

\\
case5 % WB5, and case5w are the same
& \cite{li2010small}
& 1.755e+04
& 21.80
& 1.482e+03
& 120.73
& 1.482e+03
& 25.93
& 1.482e+03   & 0.263

\\

%\\ 
%WB5 & \cite{Bukhsh2013}
%&3.264314e+02
%&7.723910e-01
%&3.264307e+02
%&2.313500e+00
%& 3.264303e+02 & 9.358597e-02  

case6ww & \cite{wood1996power}
& 3.144+03
&  0.114
&3.144e+03
& 0.74
&3.144e+03
&2.939
 & 3.144e+03 & 0.260
%\\ 
%{\bf case9mod} & \cite{Bukhsh2013}
%& {\bf 4.267068e+03}
%&7.724350e-01
%& 0.0
%&2.805652e+00
%& {\bf 3.087842e+03} & 3.937201e-01

%\\
%{\bf case9Q} & \cite{Matpower}
%&{\bf 5.301105e+03}
%&7.712110e-01
%&0.0
%&2.944110e+00
%&  
%  {\bf 5.296686e+03}   & 2.329800e-01
\\  
case14 & \cite{Matpower}
&8.082e+03
&0.201
&8.082e+03
& 0.84
%**
& --
& --
& 8.082e+03 & 0.031
 
%\\
%case14\_mod\_090 & \cite{Ghaddar2015}
%&7.806117e+03 &
%{\bf 1.154629e+00} &
%0.0 &
%3.812091e+00
%&
%  7.806115e+03 &    {\bf 2.573469e-01}
 \\
case30 & \cite{li2010small}
& 5.769e+02
& 0.788
& 5.769e+02
& 2.70
& 5.765e+02
& 6.928
& 5.769e+02   & 0.074

%\\
%case30\_mod\_m003 & \cite{Ghaddar2015}
%&5.894772e+02
%&8.165120e-01
%&0.0
%&8.600797e+00
% & 5.894770e+02   & 1.341289e+00 
%\\ 
%{\bf case30Q} & \cite{Matpower}
%& {\bf 6.230062e+02}
%& {\bf 8.266580e-01}
%&0.0
%&8.448967e+00
% & {\bf 5.768924e+02} &    {\bf 7.092309e-02}
\\ 
case39 & \cite{Matpower}
& 4.189e+04
& 0.399
& 4.189e+04 % 41889.14
& 3.26
& 4.202e+04
& 7.004
& 4.186e+04 & 0.885
\\
case57 & \cite{Matpower}
&4.174e+04
&0.674
& 4.174e+04
& 2.69
& --
& -- 
& 4.174e+04 & 0.857

\\
case57Tree
& \cite{kocuk2015new}
& 12100.86
& 6.13
& * 1.045e+04
& 4.48
& * 1.046e+04
& 342.00
& * 1.046e+04
    & 3.924

\\ 
case118 & \cite{Matpower} 
& 1.297e+05
 &  1.665
  &  1.297e+05
  &    6.57
  & --
   & --
   & 1.297e+05 & 1.967

\\
case300 & \cite{Matpower}
& 7.197e+05
& 2.410 
& --
& 17.68
& --
& -- 
& 7.197e+05    & 90.103
\end{tabular}\\[6mm]
\end{center}
*: We note that the precision here is approximately $10^{-6}$, which is the case for all three SDP-based solvers, 
Algorithm~\ref{alg:SCDM}, as well as {\tt sdp\_pf} and {\tt OPF\_Solver}.

\end{sidewaystable}

 %\item CSDP version 6.0.1, which has been developed by Borchers \cite{borchers1999csdp}
 %\item SeDuMi version 1.32, which has been developed by Sturm et al.\ \cite{sturm1999}
 %\item SDPA version 7.0, which has been developed by the SDPA group \cite{yamashita2012latest}
 %\item SDPT3 version 4.0, which has been developed by T{\"u}t{\"u}nc{\"u} et al.\ \cite{toh1999sdpt3,tutuncu2003solving}
 %\item SDPLR version 1.03 (beta), which has been developed by Burer and Monteiro \cite{BurerMonteiro2003}

%\begin{table}
\begin{sidewaystable}
\caption{For comparison, the CPU time in seconds of certain well-known general-purpose SDP solvers on the Lavaei-Low primal or dual SDP obtained from IEEE test cases and certain well-known pathological instances. Dash (--) indicates no feasible solution has been provided, often due to numerical issues or the lack of convergence. }
\label{tabComparisonWithGP}
\begin{center}
\begin{tabular}{l|l|r|r|r|r|r||r} 
   \multicolumn{2}{c|}{Instance} & \multicolumn{6}{c}{Time [s]} \\
   \hline
      Name & SDP & SeDuMi 1.32 & SDPA 7.0 & SDPLR 1.03 & CSDP 6.0.1 & SDPT3 4.0 & Mosek 7.0 \\ 
   \hline
%LMBM3LL & primal&0.4&--&0& 0.04 & 0.5 & 0.1 \\ 
%LMBM3LL & dual&0.4& 0.01 &0& 0.03 & 0.5 & 0.2 \\ 
case2w & primal&0.3& 0.003 &0& 0.02 & 0.4 & 0.1 \\ 
case2w & dual&0.3& 0.004 &0& 0.02 & 0.4 & 0.1 \\ 
case5w & primal&0.4&--&0& 0.04  & 0.5 & 0.2 \\ 
case5w & dual&0.4&--&0& 0.03 & 0.5 & 0.2 \\ 
case9 & primal&1.0& 0.05 &25& 0.1 & 0.5 & 0.3 \\ 
case9 & dual&1.0& 0.07&25& 0.1 & 0.6 & 0.3 \\ 
case14 & primal&0.7& 0.05 &58& 0.1 &-- & 0.2\\ 
case14 & dual&0.7& 0.04 &17& 0.1 &-- & 0.3\\ 
case30 & primal&2.8&--&829& 0.8 &-- & -- \\ % MOSEK Numerical problems  
case30 & dual&6.1&--&282& 1.0 &-- & --\\ % MOSEK fatal error stopenv. 
case30Q & primal&2.8&--&831& 0.8 &-- & -- \\ % MOSEK Numerical problems 
case30Q & dual&6.2&--&195& 1.0 &-- & -- \\ % MOSEK fatal error stopenv. 
case39 & primal&4.4&--&2769& 1.0 &-- & 0.7 \\ 
case39 & dual&7.7&--&723& 1.3 &-- & -- \\ % MOSEK fatal error stopenv. 
case57 & primal&3.2&--&1930& 0.5 &-- & 0.7\\ 
case57 & dual&4.0&--&1175& 1.0 &-- & 1.8\\ 
case118 & primal&10.3& 0.9 & 4400 & 3.4 &-- & 1.7 \\ 
case118 & dual&17.1&--&--& 13.0 &-- & -- \\ % MOSEK fatal error stopenv.
case300 & primal&27.5& 1.7 &--& 109.6 &-- & -- \\ % after 4.1885s, Mosek MSK_RES_TRM_STALL primal Viol.  con: 5e-03
case300 & dual&133.7&--&--& 66.7 &-- & -- \\ % MOSEK fatal error stopenv.
 \end{tabular}
 \end{center}
 %\end{table}
\end{sidewaystable}

% RIMAL_AND_DUAL_FEASIBLE  Solution status : NEAR_OPTIMAL

\clearpage

\clearpage

\begin{sidewaystable}

\caption{The results of our numerical experiments on NESTA test cases. When no feasible solution has been provided, we list the infeasibility of the least infeasible point provided in red parentheses. }
\label{tab2}
\begin{center}

% Ahoj Martin: ten OPF solver bych vynechal, pacto s normalnimi parametry nevyresi nic. 
% & \multicolumn{2}{c|}{{\tt OPF\_Solver}}  
 
\begin{tabular}{l|r|r|r|r|r|r} 
   Instance
 & \multicolumn{2}{c|}{MATPOWER}
 & \multicolumn{2}{c|}{{\tt sdp\_pf}}   
 & \multicolumn{2}{c}{Algorithm~\ref{alg:SCDM}}  
 \\ \hline 
Name  & Obj. & Time [s] &  
        Obj. &  Time [s]
  &     Obj. & Time [s]
%  &
%    Obj. & Time [s]
 \\
\hline

nesta\_case3\_lmbd
& 
5.812e+03 &
0.946  &
5.789e+03&
2.254&

 5.757e+03  &
 0.149 
\\
nesta\_case4\_gs
&
1.564e+02
&1.019
%2.139543e-23
& 
1.564e+02&
2.392&
1.564e+02
&
0.139 
\\
nesta\_case5\_pjm
& 
 {\color{red}(1.599e-01)} &
0.811
  &
 {\color{red}(1.599e-01)} &
 2.708&
2.008e+04  &
  0.216
\\
nesta\_case6\_c
& 
2.320e+01&
0.825 &
2.320e+01&
2.392&
2.320e+01  &
0.379  
\\
nesta\_case6\_ww
& 
3.143e+03&
0.884
 &
3.143e+03&
2.776
 &
  3.148e+03
  &  0.242
\\

nesta\_case9\_wscc
& 
5.296e+03 &
1.077  &
5.296e+03&
2.621&
5.296e+03  &
0.211  
\\
nesta\_case14\_ieee
& 
2.440e+02 &
0.762  &
2.440e+02&
3.164&
 2.440e+02 &
 0.267 
\\
nesta\_case30\_as        
& 
8.031e+02&
0.861&
8.031e+02&
3.916&
8.031e+02  &
 0.608 
\\
nesta\_case30\_fsr       
& 
5.757e+02&
0.898&
5.757e+02&
6.201&
 5.750e+02 &
 0.613 
\\
 
nesta\_case30\_ieee
& 
2.049e+02&
0.818&
2.049e+02&
4.335&
 2.049e+02 &
  0.788  
\\
nesta\_case39\_epri
& 
9.651e+04 &
0.863  &
9.649e+04&
5.676&
9.651e+04  &
 0.830 
\\

nesta\_case57\_ieee
& 
1.143e+03 &
1.119  &
1.143e+03&
8.053&

 1.143e+03 &
 0.957 
\\

nesta\_case118\_ieee  
& 
 {\color{red}(4.433e+02)}&
1.181 
&
 {\color{red}(4.433e+02)} &
 18.097 &
4.098e+03&
4.032
\\

nesta\_case162\_ieee\_dtc  
& 
 {\color{red}(3.123e+03)}&
0.975
 &
 {\color{red}(3.123e+03)} &
32.153
  &

4.215e+03  &
 8.328   
\\
 \hline

 nesta\_case3\_lmbd\_\_api
 & 
3.677e+02&
1.113 &

3.333e+02&
2.459 &

3.333e+02&
0.029
\\  nesta\_case5\_pjm\_\_api
 & 
 {\color{red}(5.953e+00)}
&1.002
&
 {\color{red}(5.953e+00)}
& 
2.426
&
3.343e+03  
&1.059
\\  nesta\_case6\_c\_\_api
 & 
8.144e+02&
1.034&
8.144e+02&
2.456&
8.139e+02&
0.722
\\  nesta\_case9\_wscc\_\_api
&
6.565e+02&
1.135&
 
6.565e+02&
2.869 &

 6.565e+02&
0.708 
\\

  nesta\_case14\_ieee\_\_api
 &
3.255e+02&
1.035&
3.255e+02&
3.569   &
3.245e+02 &
 0.214  
\\

  nesta\_case30\_as\_\_api
 & 
5.711e+02&
1.156&
5.711e+02&
5.433&
5.683e+02&
0.615
\\  nesta\_case30\_fsr\_\_api
 & 
3.721e+02&
0.862&
3.309e+02&
5.276&
2.194e+02&
0.604
\\  nesta\_case30\_ieee\_\_api
 & 
4.155e+02&
1.076&
4.155e+02&
5.293&
 4.155e+02&
0.618
\\  nesta\_case39\_epri\_\_api
 & 
 {\color{red}(1.540e+02)}&
0.965
 &
 {\color{red}(1.540e+02)}
&5.109
&
7.427e+03&
1.978
\\  nesta\_case57\_ieee\_\_api
 & 
  1.430e+03&
1.041
&1.429e+03&
7.625 
&
1.429e+03&
2.437

\\

 nesta\_case162\_ieee\_dtc\_\_api
 & 
 {\color{red}(1.502e+03)}&
1.215  &
 
 {\color{red}(1.502e+03)}&
43.339 & 
6.003e+03&
5.833
\\
 
 \end{tabular}
 \end{center} 
 
\end{sidewaystable}

\subsection{Adaptive Updates of $\mu$}
\label{sec:adaptivemu}

As it has been stressed above, it is important to pick an appropriate $\mu$.
Alternatively, one can adapt a strategy to change 
$\mu$ adaptively. 
This can be achieved by requiring the infeasibility to shrink by a fixed factor between consecutive iterations. 
If such an infeasibility shrinkage is not feasible, one ignores the proposed iteration update and decreases $\mu$.
Figures \ref{fig:adaptiveMu1} and \ref{fig:adaptiveMu2} show the evolution of 
objective function value and the infeasibility for 400 iterations with 
fixed values \EDIT{ $\mu\in \{10^{-1}, 10^{-2}, 10^{-3}, 10^{-4}, 10^{-5}, 10^{-6}\}$ } and for the adaptive strategy
on the instance case5.
Figure \ref{fig:adaptiveMu3} details the evolution of 
$\mu$ within the adaptive strategy \EDIT{on the same instance}.
One can see that choosing very small $\mu$ forces the algorithm to ``jump'' close to a feasible point, and get ``stuck''.
On the other hand, larger fixed values of $\mu$ lead to the same objective value.
One can also see that the adaptive strategy is rejecting the first 80 updates, until the value of $\mu$ is small enough; subsequently, it starts to make rapid progress.

\begin{figure}
\centering
\includegraphics[width=3in]{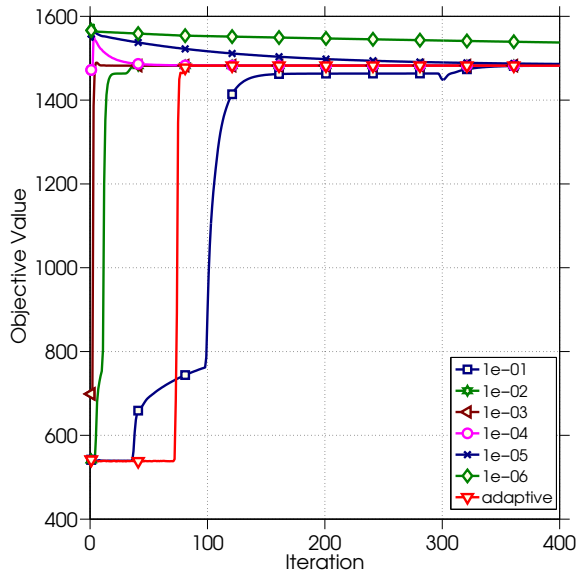}
\caption{The evolution of the objective function value \EDIT{over time on instance case5}, for a number of choices of $\mu$ and the adaptive updates of $\mu$.}
\label{fig:adaptiveMu1}
\end{figure}

\begin{figure}
\centering
\includegraphics[width=3in]{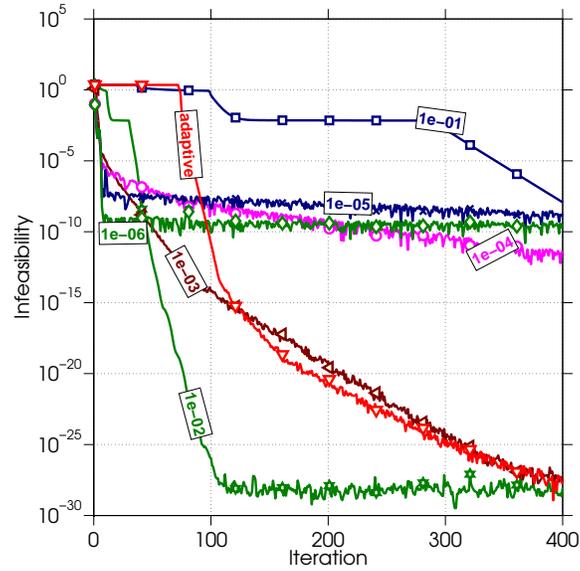}
\caption{The evolution of the infeasibility \EDIT{over time on instance case5}, for a number of choices of $\mu$ and the adaptive updates of $\mu$.}
\label{fig:adaptiveMu2}
\end{figure}

\begin{figure}
\centering
\includegraphics[width=3in]{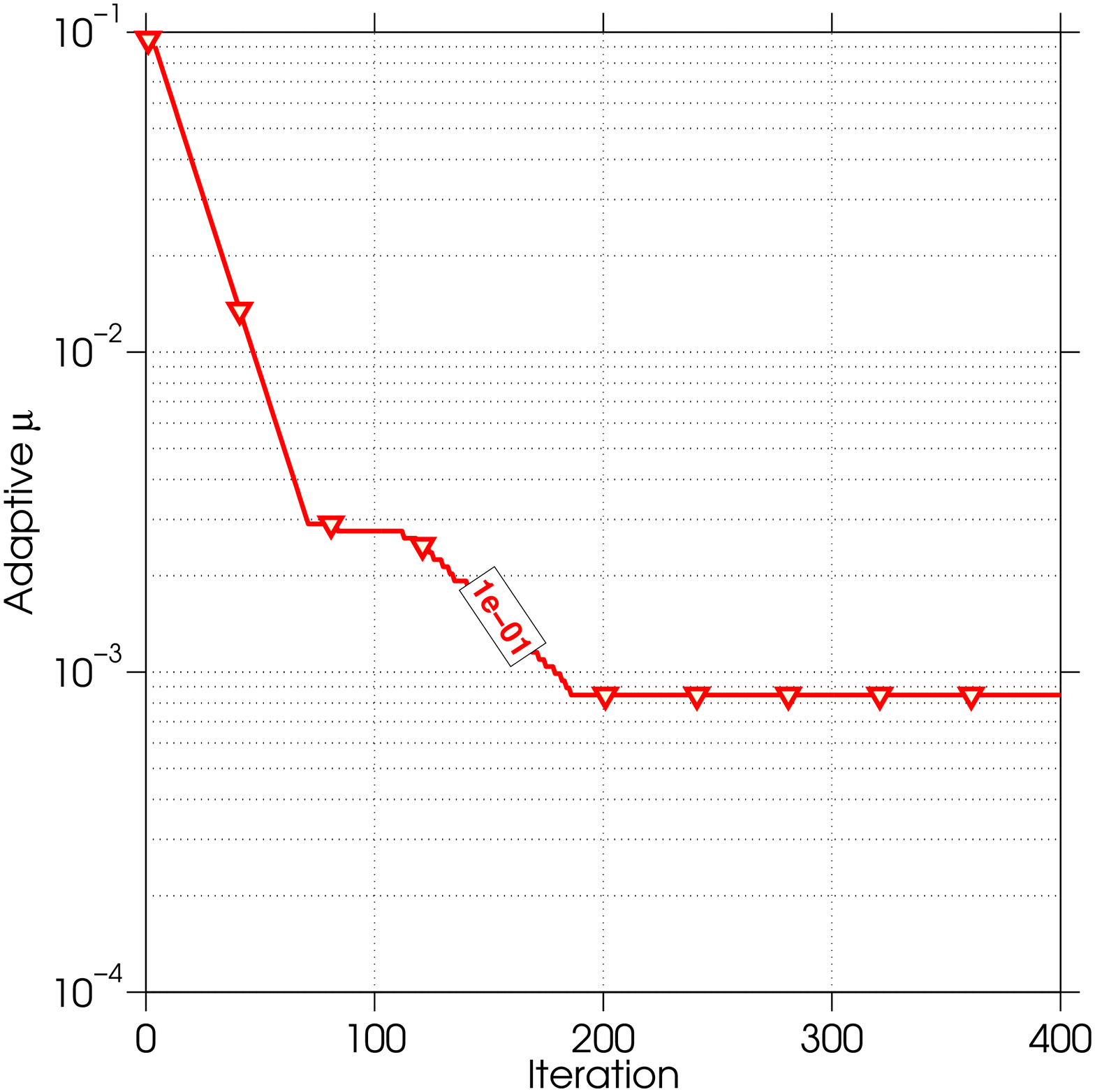}
\caption{The evolution of $\mu$ with adaptive updates \EDIT{over time on instance case5}.}
\label{fig:adaptiveMu3}
\end{figure}

\subsection{Large Instances}

We have also experimented with the so called Polish \EDIT{network} \cite{Matpower},
\EDIT{also known} as case2224 and case2736sp, 
as well as the instances collected by the Pegase project \cite{6465783}, 
such as case9241peg.
There, we cannot obtain solutions with the same precision as Matpower, 
within comparable run-times. 
However, we can decrease the initial infeasibility
by factor of $10^2$ for case2224 in 200 seconds,
by factor of $10^6$ for case2736sp in 200 seconds, and 
by factor of $10^4$ for case9241peg, again in 200 seconds.
Although these results are not satisfactory, yet, they may be useful, when a good initial solution is available.
As discussed in Section~\ref{conclusions}, we aim to improve upon these results by using 
a two-step method, where first-order methods on the convex problem are 
combined with a second-order method on the non-convex problem.
These results also motivate the need for parallel computing. %as listed in Section~\ref{sec:parallel},
%and the use of methods with better rates of convergence.

\subsection{Parallel Computing}
\label{sec:parallel}

As it has been discussed above, Algorithm~\ref{alg:SCDM} is easy to implement in a parallel fashion. 
We have implemented our multi-threading variant using OpenMP, in order to achieve portability. 
The significant overhead of using OpenMP makes it impossible to obtain a speed-up on very small instance.
For example, in case3w,
there is not enough work to be shared across (any number of) threads to offset the overhead.
In Figure~\ref{fig:scaling}, we present the scaling properties of Algorithm \ref{alg:SCDM}. 
In particular, we illustrate the speed-up (the inverse of the multiple of single-threaded run-time) as a function of the number of cores
used for a number of larger instances. 
We can also observe that even case118 is too small to benefit from a considerable speed-up.
On the other hand, for large datasets, such as the Polish case2224 or larger, 
we observe a significant speed-up. 

\begin{figure}
\centering
\includegraphics[width=3in]{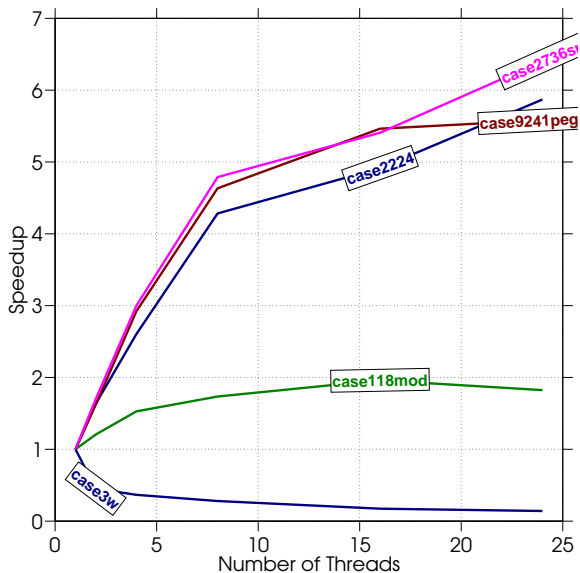}
\caption{The speed-up as a function of the number of cores employed.}
\label{fig:scaling}
\end{figure}
  
\section{Conclusions}
\label{conclusions}

Our approach seems to bring the use of SDP relaxations of ACOPF 
closer to the engineering practice.
A number of authors have recently explored
elaborately constructed linear programming 
\cite{Bienstock2015}
and second-order cone programming 
\cite{7285718,MolzahnHiskens2015a,MolzahnHiskens2015b,
kocuk2015new,kocuk2016strong} 
relaxations, many \cite{7285718,kocuk2016strong} of which are not comparable to the SDP relaxations in the sense that they may be stronger or weaker, depending on the instance,
but aiming to be solvable faster.
Algorithm~\ref{alg:SCDM}
suggests that there are simple first-order algorithms, which can solve SDP relaxations of ACOPF faster than previously,
at least on some instances.

A major advantage of first-order methods over second-order methods \cite{Matpower,sturm1999} is the ease of their parallelisability. 
This paper presents considers a symmetric multi-processing, where
memory is shared, but a distributed variant, 
where $\K$ agents perform the iterations, 
is clearly possible, c.f. \cite{marecek2015distributed}. 
The $\K$ agents may represent companies, each of whom owns some of the generators and does not want to expose the details, such as cost functions.
If $\K$ agents are $\K$ computers, a considerable speed-up can be obtained.
Either way, power systems analysis could benefit from parallel and distributed computing.

The question as to how far could the method scale, remains open.
In order to improve the scalability, one
could try to combine first- and second-order methods \cite{liu2015hybrid}.
We have experimented with an extension, which uses 
Smale's $\alpha$-$\beta$ theory \cite{Blum1997} to stop the computation
at the point $z_0$, where we know, based on the analysis of the
Lagrangian and its derivatives, 
that a Newton method or a similar algorithm with quadratic speed of convergence \cite{chen1994approximate}
will generate sequence $z_i$ to the correct optimum $z^*$, i.e.
\begin{align}
|z_i - z^* | \le (1/2)^{2^i - 1} |z_0 - z^*|.
\end{align}
This should be seen as convergence-preserving means of auto-tuning of the switch to a second-order method for convex functions. 
%which allows for tap-changing and phase-shifting transformers, paralllel lines, and multiple generators at one bus, among others, 
%which converges across all of Polish instances. 

One should also study infeasibility detection.
\EDIT{Considering the test of feasibility of an SDP is not known to be in NP \cite{ramana1997exact},}
we assumed the instance are feasible, throughout.
\EDIT{Nevertheless, one may need to test their feasibility first, in many practical applications.}
There are some very fast heuristics \cite{6629370}
available already, but one could also use Lagrangian methods in a two-phase
scheme, common in robust linear-programming solvers.
Both in methods based on simplex and feasible interior-point,
one first considers a variant of the problem, 
which has constraints relaxed by the addition of slack variables, 
with objective minimising a norm of the slack variables.
This makes it possible to find feasible solutions quickly,
and by duality, one can detect infeasibility.

One could also apply additional regularisations, following \cite{6736636,7402079}.
In \ref{R1}, one could drop the rank-one constraint and
modify the objective function to penalise the solutions
with large ranks, e.g., by adding the term $\lambda \|W\|^*$ to the objective function, where 
$\|\cdot\|_*$ is a nuclear norm \cite{fazel2002matrix}
and $\lambda>0$ is a parameter.
Alternatively, one can replace the rank constraint by the requirement that the nuclear norm of the matrix should be small, i.e. $\|W\|_* \leq \lambda$.
However, both approaches require a search for a suitable parameter $\lambda$ such that the optimal solution has indeed rank 1.
Moreover, the penalised alternative may not produce an optimal solution of \ref{PP4}, necessitating further algorithmic work.

Finally, one could extend the method to solve relaxations
a variety of related applications, such as 
security-constrained problems, 
stability-constrained problems, 
network expansion planning \cite{marecek2016minlp}, 
and unit commitment problems. 
The question as to whether the method could generalise to the higher-order relaxations, c.f., Ghaddar et al. \cite{Ghaddar2015}, also remains open.
%When the \ref{LL} relaxation does not have zero duality gap. 
First steps \cite{Ma2015PhD} have been taken, but much work remains to be done. % Ma2015 

\bibliographystyle{gDEA}
\bibliography{acopf,literature}

\end{document}